\documentclass{amsart}
\usepackage{amsmath, amscd, amssymb, amsthm}
\usepackage{mathrsfs}
\usepackage{bbm}
\usepackage{latexsym}
\usepackage{amsfonts}
\usepackage{graphicx}
\usepackage{needspace}
\usepackage[all,cmtip]{xy}
\usepackage[
  colorlinks=true,
  linkcolor=blue,
  citecolor=blue,
  urlcolor=blue,
  anchorcolor=blue,
  breaklinks=true,
  pagebackref
]{hyperref}

\hypersetup{
  pdftitle={Hermitian threefolds with constant holomorphic sectional curvature},
  pdfauthor={Shuwen Chen and Fangyang Zheng}
}

\usepackage{geometry}
\numberwithin{equation}{section}

\newtheorem{theorem}{Theorem}
\newtheorem{lemma}{Lemma}
\newtheorem{corollary}[theorem]{Corollary}
\newtheorem{remark}{Remark}
\newtheorem{proposition}{Proposition}

\newtheorem{conjecture}{Conjecture}

\newcommand{\C}{\mathbb C}
\newcommand{\R}{\mathbb R}

\newcommand{\B}{\mathcal B}
\newcommand{\At}{\mathcal A}
\newcommand{\Nt}{\mathcal N}
\newcommand{\KM}{K_M}
\newcommand{\LL}{\mathcal L}
\newcommand{\PP}{\mathcal P}
\newcommand{\KB}{\mathsf K}

\renewcommand*\backref[1]{}
\renewcommand*\backrefalt[4]{ \ifcase #1 \or (cited on page #2) \else (cited on pages #2) \fi}

\providecommand{\eps}{}
\renewcommand{\eps}{\varepsilon}
\newcommand{\be}{\begin{equation}}
\newcommand{\ee}{\end{equation}}
\newcommand{\bea}{\begin{eqnarray}}
\newcommand{\eea}{\end{eqnarray}}

\newcommand{\bp}{\bar\partial}
\newcommand{\ip}[2]{(#1,#2)}

\newcommand{\step}[1]{\par\smallskip\noindent\emph{#1}\ }

\DeclareMathOperator{\tr}{tr}
\DeclareMathOperator{\re}{Re}
\DeclareMathOperator{\im}{Im}
\DeclareMathOperator{\rank}{rank}
\DeclareMathOperator{\Sym}{Sym}
\DeclareMathOperator{\Ann}{Ann}
\DeclareMathOperator{\adj}{adj}
\DeclareMathOperator{\divg}{div}
\DeclareMathOperator{\Curl}{Curl}
\DeclareMathOperator{\diag}{diag}

\DeclareMathOperator{\Span}{span}

\allowdisplaybreaks[2]
\def\XXint#1#2#3{{\setbox0=\hbox{$#1{#2#3}{\int}$ }
\vcenter{\hbox{$#2#3$ }}\kern-.6\wd0}}

\begin{document}

\title[Hermitian threefolds with constant holomorphic sectional curvature]
{Hermitian threefolds with constant holomorphic sectional curvature}

\author{Shuwen Chen}
\address{Shuwen Chen. School of Mathematical Sciences, Chongqing Normal University, Chongqing 401331, China}
\email{{3153017458@qq.com}}\thanks{Zheng is the corresponding author. Zheng is partially supported by National Natural Science Foundations of China
with the grant No. 12471039 and  12141101, and is supported by the 111 Project D21024.}

\author{Fangyang Zheng}
\address{Fangyang Zheng. School of Mathematical Sciences, Chongqing Normal University, Chongqing 401331, China}
\email{20190045@cqnu.edu.cn; franciszheng@yahoo.com} \thanks{}

\date{}

\begin{abstract}
An old conjecture in non-K\"ahler geometry states that if the Chern holomorphic sectional curvature of a compact Hermitian manifold is equal to a  constant $c$, then the metric  must be  K\"ahler when $c\neq 0$ and be Chern flat  when $c=0$. The conjecture is known to be true in dimension two by the work of Balas--Gauduchon and Apostolov--Davidov--Mu\v{s}karov in the 1980s and 1990s. Recently, Qin and Tian proved the conjecture in complex dimension three when $c\neq 0$, and the $c=0$ case was proved by Chen--Li under the additional assumption that the metric is balanced. In this article we remove this additional hypothesis and complete the confirmation of the conjecture in complex dimension three. The proof relies heavily on the fact that in dimension three the torsion $3$-tensor can be equivalently expressed as a $2$-tensor twisted by the canonical line bundle, plus Bochner type integration formulas. In particular,  the method cannot be directly generalized to higher dimensions.
\end{abstract}

\maketitle

\markleft{Shuwen Chen and Fangyang Zheng}
\markright{Hermitian threefolds with constant holomorphic sectional curvature}

\tableofcontents

\section{Introduction and statement of results}\label{intro}

A long-lasting open question in non-K\"ahler geometry is the following:

\begin{conjecture} \label{conj1}
Let $(M^n,g)$ be a compact Hermitian manifold. Assume its Chern holomorphic sectional curvature is equal to a constant $c$. Then the metric $g$ must be K\"ahler when $c\neq 0$ and must be Chern flat when $c=0$.
\end{conjecture}

For the history and related developments of the conjecture, we refer to the survey \cite{Zheng}.

The conjecture is known to be true in dimension 2 by the work of Balas and Gauduchon \cite{Balas, BG} in 1985 when the constant is zero or negative, and by Apostolov, Davidov, and Mu\v{s}karov \cite{ADM} in 1996 when the constant is positive. In higher dimensions, the conjecture has been known in a number of special cases. For instance, it was confirmed for all twistor spaces by Davidov, Grantcharov, and Mu\v{s}karov \cite{DGM}, for all Chern K\"ahler-like manifolds by Tang \cite{Tang}, for all Bismut K\"ahler-like manifolds by Rao and Zheng \cite{RaoZ}, for all complex nilmanifolds by Li and Zheng \cite{LiZ}, for all Fujiki class ${\mathcal C}$ manifolds when $c\leq 0$ by Qin and Tian \cite{QinTian}, and for unimodular Hermitian Lie algebras containing an abelian ideal of real codimension two in our work \cite{ChenZcodim2}. In the important case of locally conformally K\"ahler manifolds, the conjecture was confirmed by H. Chen, L. Chen, and Nie \cite{CCN} in 2021 for the $c\leq 0$ cases, while the remaining $c>0$ case was completed by Huang and Wan in their recent work \cite{HW}.

Another important class consists of Bismut torsion-parallel (BTP) manifolds, whose Bismut connection (also known as Strominger connection in some literature, \cite{Bismut, Strominger}) has parallel torsion.
In \cite{ChenZ26}, we confirmed the conjecture for compact
non-balanced BTP manifolds and compact balanced BTP threefolds.
Wang and Zheng \cite{WangZ} established the balanced BTP fourfold case. In arbitrary dimensions, H.~Wang \cite{HWang} established the $c\neq0$ case for compact balanced BTP manifolds, while Wang and Zheng \cite{WangZ2} confirmed the conjecture for compact balanced BTP manifolds for all values of $c$. We also studied related questions for the Bismut  connection and more general canonical metric connections in \cite{ChenZ,ChenZheng}; see also the work of Chen and Nie \cite{ChenNie}. In \cite{ChenZparallel}, we further showed that a Hermitian connection with parallel torsion and nonzero constant holomorphic sectional curvature is torsion-free; consequently, the underlying Hermitian metric is K\"ahler. This result requires neither compactness nor completeness.

Related results have also been obtained for real bisectional curvature and pluriclosed metrics. Zhou and Zheng \cite{ZhouZ} proved that every compact Hermitian threefold with vanishing real bisectional curvature is Chern flat. Broder and Tang \cite{BroderTang} proved that pluriclosed metrics with vanishing Chern holomorphic sectional curvature on compact K\"ahlerian manifolds are flat K\"ahler metrics. They also established K\"ahler flatness for Hermitian metrics with vanishing real bisectional curvature on compact manifolds in Fujiki's class $\mathcal C$. More recently, Tang \cite{Tang26} proved that a compact Hermitian manifold of complex dimension at least two with constant real bisectional curvature has zero constant and is Chern flat. He also proved that a pluriclosed metric with constant negative Chern holomorphic sectional curvature on a compact K\"ahlerian manifold is K\"ahler.

Specializing in complex dimension three, Ma and Nie \cite{MaN} proved the conjecture for compact normal balanced threefolds with $c\leq0$, and Chen and Li \cite{ChenLi} subsequently removed the normality assumption in this range. More recently, Qin and Tian \cite{QinTian} confirmed the conjecture for $c\neq0$ on arbitrary compact Hermitian threefolds.

The main purpose of this article is to establish the $c=0$ case in dimension three:

\begin{theorem} \label{thm1}
Let $(M^3,g)$ be a compact Hermitian threefold. Assume that its Chern holomorphic sectional curvature is identically zero. Then $g$ is Chern flat.
\end{theorem}

Recall that compact Chern flat manifolds are exactly compact quotients of complex Lie groups equipped with left-invariant metrics, by the classic work of Boothby \cite{Boothby}. In particular, in complex dimension three, one can write down the explicit list of all (connected and simply connected) complex Lie groups of complex dimension three.

Combining Theorem \ref{thm1} with the aforementioned result of Qin and Tian \cite{QinTian} for the $c\neq 0$ case, we get the affirmation of the conjecture in complex dimension three:

\begin{corollary}
Conjecture \ref{conj1} holds in complex dimension three.
\end{corollary}

Let us briefly describe the proof of Theorem \ref{thm1}, which relies heavily on the fact that we are in complex dimension three. Denote by $E=T^{1,0}M$ the holomorphic tangent bundle of the given Hermitian threefold $(M^3,g)$. The Chern torsion tensor $T$ is a linear map from $\Lambda^2E$ into $E$, hence a section of the bundle $\mbox{Hom}(\Lambda^2E, E)\cong \Lambda^2E^{\ast}\otimes E$. Since $\dim M=3$, we have a bundle isomorphism
\begin{equation}
 \Lambda^2E^{\ast} \cong E \otimes \Lambda^3E^{\ast} = E \otimes K_M,
 \end{equation}
where $K_M$ denotes the canonical line bundle of $M^3$. Therefore $T$ is equivalent to a section ${\mathcal A}$ of the bundle $E\otimes E\otimes K_M$, namely, a $K_M$-valued covariant $2$-tensor on $M$. Let ${\mathcal B}=\frac12 ( {\mathcal A} + \mbox{}^t\!{\mathcal A})$ be the symmetric part of ${\mathcal A}$. Then ${\mathcal B}$ is a section of $\Sym^2\!E \otimes K_M$, which contains a useful part of information of the Chern torsion tensor $T$. This $K_M$-valued symmetric $2$-tensor ${\mathcal B}$ will be a major tool in our proof. Another major tool is a tensor $Q$ formed by the Chern curvature tensor, which is a linear map from $\Sym^2\!\bar{E} $ into $E \otimes K_M$, and we will refer to $Q$ as the {\em curvature homomorphism}. The Bianchi identities give differential equations for these tensors, and two integral identities then provide a scalar inequality and an estimate which will be used repeatedly.

Assume that a compact Hermitian threefold $(M^3,g)$ has vanishing Chern holomorphic sectional curvature.  The first reduction shows that if $g$ is not Chern flat, then the rank of $Q$, which is the maximum for all $x\in M$ of the dimension of the image space of $Q_x$, is at most $2$, and the rank cannot be one.  The next step is to deal with the rank zero case, namely the $Q=0$ case. Note that this case does require a proof as $Q=0$ does not mean $R=0$ yet because $Q$ does not encode all the curvature components. Now we are left with the rank two case. In this case, we will show that the rank-two locus form an open dense subset in $M$, in which we obtain an invariant plane and a line determined by $\partial\eta$, where $\eta$ is Gauduchon's torsion one-form. We then divide the proof into two cases, depending on whether
$\det\B\not\equiv0$ or $\det\B\equiv0$. Both cases will lead to contradictions:
in the first case an adapted frame leads to two incompatible evaluations of the same curvature component, while in the second case a holomorphic quotient gives a degenerate binary quadratic form, from which we obtain the pointwise estimate needed to close the integral argument.

The paper is organized as follows. In \S\ref{sec:prelim} we fix the notations and derive the local identities. The integral formulas are established in \S\ref{sec:integral}. Section~\ref{sec:rank} treats the curvature-rank reduction, including the case $Q=0$. The geometry of the rank-two locus is studied in \S\ref{sec:geometry}. Sections~\ref{sec:full} and~\ref{sec:singular} treat the two determinant alternatives. The proof of Theorem~\ref{thm1}, together with the treatment of exceptional sets, is assembled in \S\ref{sec:conclusion}.

\vspace{0.3cm}

\section{Preliminaries}\label{sec:prelim}

Let $(M^3,g)$ be a Hermitian manifold. We will follow the notations of \cite{YangZ} and write $g=\langle\,,\,\rangle$ for its complex bilinear extension, and put $\ip XY=\langle X,\bar Y\rangle$ for the Hermitian product on $E=T^{1,0}M$. Thus this product is linear in its first argument. We use the same convention on the induced tensor bundles. Denote by $\nabla$, $T$, and $R$ the Chern connection, its torsion, and its curvature.  Our conventions are
\[
T(X,Y)=\nabla_XY-\nabla_YX-[X,Y],\qquad
R(X,Y)=[\nabla_X,\nabla_Y]-\nabla_{[X,Y]}.
\]

Let $e=\{e_1,e_2,e_3\}$ be a local unitary frame, and let $\varphi=\{\varphi_1,\varphi_2,\varphi_3\}$ be its dual coframe. Following the unhalved torsion convention used in \cite{ChenZheng,ZhaoZheng}, write
\begin{equation}\label{eq:torsion-convention}
T(e_i,e_k)=\sum_jT^j_{ik}e_j,\qquad
R_{i\bar jk\bar\ell}=\ip{R(e_i,\bar e_j)e_k}{e_\ell}.
\end{equation}
If $\theta$ and $\Theta$ are the Chern connection and curvature matrices in the convention $\nabla e_i=\sum_j\theta_{ij}e_j$, then
\begin{equation}\label{eq:structure}
d\varphi=-{}^t\theta\wedge\varphi+\tau,\qquad
\Theta=d\theta-\theta\wedge\theta,\qquad
\tau_j=\frac12\sum_{i,k}T^j_{ik}\varphi_i\wedge\varphi_k.
\end{equation}
In block-matrix calculations it will also be convenient to use the column-coefficient matrix $\Omega={}^t\!\theta$. Its curvature matrix is ${}^t\!\Theta=d\Omega+\Omega\wedge\Omega$. These are two matrix conventions for the same connection. The latter convention will always be specified when used.

Following the notions of \cite{ZhaoZ23, ZhaoZheng}, an index after a comma denotes a covariant derivative with respect to $\nabla$. In particular, $T^\ell_{ik,\bar j}=(\nabla_{\bar e_j}T)^\ell_{ik}$. Derivatives of a tensor with additional derivative indices act on those indices as well. We denote the canonical bundle by $\KM=\det E^*$ and put
\begin{equation}\label{eq:eta-rho}
\eta=\sum_i\eta_i\varphi_i,\quad \eta_i=\sum_kT^k_{ki},\quad
\Phi_{i\bar j}=\eta_{i,\bar j},\quad
\rho_{i\bar j}=\sum_kR_{i\bar jk\bar k},\quad
\rho^{(2)}_{i\bar j}=\sum_kR_{k\bar ki\bar j},
\end{equation}
where $\eta$ is Gauduchon's torsion one-form determined by $\partial\omega^2=-\eta\wedge\omega^2$, and $\rho$ is the first Chern--Ricci tensor. The associated fundamental form and volume are respectively
\[
\omega=\sqrt{-1}\sum_i\varphi_i\wedge\bar\varphi_i,\qquad dV=\frac{\omega^3}{6}.
\]

Unless otherwise indicated, all indices range from $1$ to $3$ and repeated indices are summed. The squared norm of a tensor is the sum of the squares of all its ordered unitary components. On symmetric tensors this is the induced subspace norm, without a factorial divisor. On exterior powers we use the exterior metric, so that the squared norm of a two-form is one half of its ordered two-index squared norm. We use $\Sym_3$ for the average over all six permutations and $\odot$ for normalized symmetric multiplication. The superscript $*$ denotes Hermitian adjoint, and ${}^t$ denotes transpose without conjugation.

\emph{For the remainder of the paper we assume the Chern holomorphic sectional curvature vanishes identically.} By polarization, this is equivalent to
\begin{equation}\label{eq:polarization}
R_{i\bar jk\bar\ell}+R_{k\bar ji\bar\ell}
+R_{i\bar\ell k\bar j}+R_{k\bar\ell i\bar j}=0.
\end{equation}

We start with the following
\begin{lemma}\label{lem:ricci}
Under the above assumption,
\begin{equation}\label{eq:ricci-conversion}
\rho^{(2)}=-3\rho+\Phi+\Phi^*,\qquad
\chi:=\tr\Phi=2\tr\rho\in\R.
\end{equation}
\end{lemma}
\begin{proof}
In local holomorphic coordinates the Chern coefficients satisfy
$\Gamma^\ell_{ik}g_{\ell\bar q}=\partial_i g_{k\bar q}$ and $\nabla_{\bar j}\partial_k=0$. Consequently,
\begin{equation}\label{eq:mixed-bianchi}
T^\ell_{ik,\bar j}=\partial_{\bar j}(\Gamma^\ell_{ik}-\Gamma^\ell_{ki})
=R_{k\bar ji}{}^\ell-R_{i\bar jk}{}^\ell.
\end{equation}
At the point under consideration take a unitary frame, and set
\[
C_{i\bar j}=\sum_kR_{k\bar ji\bar k},\qquad
D_{i\bar j}=\sum_kR_{i\bar kk\bar j}.
\]
Contracting \eqref{eq:mixed-bianchi} gives $\Phi=\rho-C$. Curvature reality,
$\overline{R_{i\bar jk\bar\ell}}=R_{j\bar i\ell\bar k}$, gives $D=C^*=\rho-\Phi^*$. On the other hand, setting $\ell=k$ in \eqref{eq:polarization} and summing yields
\[
\rho+C+D+\rho^{(2)}=0.
\]
This proves the first formula. The traces of $\rho$ and $\rho^{(2)}$ agree by relabeling the two summed indices. Moreover, $\tr C=\tr D=\overline{\tr C}$, so $\chi$ is real. Taking the trace now gives the second formula.
\end{proof}

Throughout this article, we will write $\varepsilon_{ijk}=0$ whenever $\{ i,j,k\} \neq \{ 1,2,3\}$, and
$$ \varepsilon_{123}=\varepsilon_{231}=\varepsilon_{312}=1, \varepsilon_{213}=\varepsilon_{321}=\varepsilon_{132}=-1.$$
As we mentioned in the previous section, using the above alternating symbol $\eps_{ijk}$ together with the local unit section $\nu=\varphi_1\wedge\varphi_2\wedge\varphi_3$ of $\KM$, we can identify the Chern torsion $T$ with a $K_M$-valued $2$-tensor $\At$. Set
\begin{equation}\label{eq:AB-def}
\At_{a\ell}=\frac12\sum_{i,k}\eps_{ika}T^\ell_{ik},\qquad
\B=\frac12(\At+{}^t\!\At).
\end{equation}
So $\B$ is a global section of $\Sym^2\!E\otimes\KM$. We remark that $\B$ is not the nonnegative quadratic torsion tensor commonly denoted by $B$ in \cite{ZhaoZ23, ZhaoZ24, ZhaoZheng}, and we use the calligraphic notation here. All alternating-symbol identifications below retain the canonical factor. Put
\begin{equation}\label{eq:QSG-def}
\begin{split}
D_{a\ell j}&=\B_{a\ell,\bar j},\qquad S=\Sym_3D,\qquad G=\Phi^*-2\rho,\\
Q_{a\ell j}&=\frac14\sum_{i,k}\eps_{ika}
\bigl(R_{i\bar jk\bar\ell}+R_{i\bar\ell k\bar j}\bigr).
\end{split}
\end{equation}
Here the barred derivative index in $D$ is raised by the metric, so $S\in\Sym^3E\otimes\KM$. We regard $Q$ as a homomorphism $\Sym^2\!\bar E\longrightarrow E\otimes\KM$. Its rank will always mean the rank of this homomorphism, rather than the matrix rank of one of its quadratic components.

\begin{lemma}\label{lem:algebra}
The tensors above satisfy
\begin{align}
&T^\ell_{ik}=\sum_a\eps_{ika}\B_{a\ell}
+\frac12(\delta_{i\ell}\eta_k-\delta_{k\ell}\eta_i),
\quad \quad  |T|^2=2|\B|^2+|\eta|^2,\label{eq:torsion-split}\\
&S=-2\Sym_3Q, \quad \quad  \tr G=0,\label{eq:SG}\\
& G_{rj}=2\sum_{a,\ell}\eps_{a\ell r}Q_{a\ell j},\qquad
Q_{a\ell j}=-\frac12S_{a\ell j}
+\frac16\sum_b(\eps_{a\ell b}G_{bj}+\eps_{ajb}G_{b\ell}).\label{eq:Q-decomp}
\end{align}
At any point where $\rank Q\leq1$, one has
\begin{equation}\label{eq:rank-one-bound}
|\tr(G^2)|\leq3|S|^2.
\end{equation}
At such a point $Q\ne0$ implies $S\ne0$.
\end{lemma}
\begin{proof}
The identity $\sum_{i,k}\eps_{ika}\eps_{ikb}=2\delta_{ab}$ gives
$T^\ell_{ik}=\sum_a\eps_{ika}\At_{a\ell}$. The definition of $\eta$ gives
\[
\eta_1=\At_{23}-\At_{32},\quad
\eta_2=\At_{31}-\At_{13},\quad
\eta_3=\At_{12}-\At_{21},
\]
and hence $\At_{a\ell}=\B_{a\ell}+\frac12\sum_r\eps_{a\ell r}\eta_r$. Substitution proves the first identity in \eqref{eq:torsion-split}. The symmetric and skew-symmetric parts of $\At$ are orthogonal, and the squared norm of its skew part is $|\eta|^2/2$. This proves the norm identity. By \eqref{eq:mixed-bianchi}, we have
\begin{equation}\label{eq:D-from-R}
D_{a\ell j}=-\frac12\sum_{i,k}
\bigl(\eps_{ika}R_{i\bar jk\bar\ell}+\eps_{ik\ell}R_{i\bar jk\bar a}\bigr).
\end{equation}
Taking the full symmetrization gives $S=-2\Sym_3Q$. Let us write
$\widetilde G_{rj}=2\sum_{a,\ell}\eps_{a\ell r}Q_{a\ell j}$. Since $Q_{a\ell j}=Q_{aj\ell}$, the trace of $\widetilde G$  is zero. Contracting the two alternating symbols gives
\[
\widetilde G_{rj}=\frac12(C_{r\bar j}-\rho_{r\bar j}
+\rho^{(2)}_{r\bar j}-D_{r\bar j})=\Phi^*_{r\bar j}-2\rho_{r\bar j},
\]
where the last equality uses Lemma~\ref{lem:ricci} and the contractions in its proof. For a trace-free matrix $G$, put
\[
J(G)_{a\ell j}=\frac16\sum_b(\eps_{a\ell b}G_{bj}+\eps_{ajb}G_{b\ell}).
\]
This tensor is symmetric in $\ell,j$, has zero full symmetrization, and satisfies
$2\sum_{a,\ell}\eps_{a\ell r}J(G)_{a\ell j}=G_{rj}$. Thus
$V=Q+S/2-J(G)$ has zero full symmetrization and zero alternating-symbol contraction. The latter makes $V$ symmetric in its first two positions; it is already symmetric in its last two. Hence it is fully symmetric and therefore zero. This proves \eqref{eq:Q-decomp}.

It remains to prove \eqref{eq:rank-one-bound}. When $Q\ne0$ and its output has dimension one, choose a unitary frame with that output along the first axis. After including the unit canonical factor, write
\[
Q_{1\ell j}=q_{\ell j},\qquad Q_{2\ell j}=Q_{3\ell j}=0,
\]
where $q$ is symmetric. Formula~\eqref{eq:Q-decomp} gives
\[
G_{1j}=0,\quad G_{2j}=-2q_{3j},\quad G_{3j}=2q_{2j},\qquad
\tr(G^2)=8(q_{23}^2-q_{22}q_{33}).
\]
The possibly nonzero symmetric components of $S$ are
\[
\begin{gathered}
S_{111}=-2q_{11},\quad S_{112}=-\frac43q_{12},\quad
S_{113}=-\frac43q_{13},\\
S_{122}=-\frac23q_{22},\quad S_{123}=-\frac23q_{23},\quad
S_{133}=-\frac23q_{33}.
\end{gathered}
\]
Their ordered multiplicities are $1,3,3,3,6,3$, respectively. Therefore
\[
|S|^2=4|q_{11}|^2+\frac{16}{3}(|q_{12}|^2+|q_{13}|^2)
+\frac43(|q_{22}|^2+|q_{33}|^2)+\frac83|q_{23}|^2.
\]
This is positive when $q\ne0$, and
\[
|\tr(G^2)|\leq8|q_{23}|^2+8|q_{22}||q_{33}|
\leq8|q_{23}|^2+4|q_{22}|^2+4|q_{33}|^2\leq3|S|^2.
\]
The case $Q=0$ follows at once from \eqref{eq:SG} and \eqref{eq:Q-decomp}. For abstract arrays the constant is sharp, as is seen by taking $q=\diag(0,1,-1)$. This completes the proof of the lemma.
\end{proof}

For our later use, let us introduce the following notations:
\begin{equation}\label{eq:N-w-def}
\Nt_{abp}=\B_{ab,p}-\frac12\eta_p\B_{ab},\qquad
W=\partial\eta,\qquad w_a=\frac12\sum_{i,k}\eps_{ika}W_{ik}.
\end{equation}
Thus $w$ is a section of $E\otimes\KM$. Let us also write
\[
I_{\B}(v)_{abc}=\frac13(\B_{ab}v_c+\B_{ac}v_b+\B_{bc}v_a).
\]

\begin{lemma}\label{lem:differential}
Under the above notation,
\begin{align}
& Q_{a\ell j,p}=\frac12\eta_pQ_{a\ell j}
+\sum_{b,t}\eps_{pbt}\B_{at}Q_{b\ell j},\label{eq:Q-transport}\\
& S_{abc,p}=\frac12\eta_pS_{abc}
-2\Sym_{abc}\sum_{u,v}\eps_{puv}\B_{av}Q_{ubc},\label{eq:S-transport}\\
& \Sym_{abc}\Nt_{abp,\bar c}=-\frac12 I_{\B}((G^*)_p)_{abc},
\quad \quad  \sum_p\Nt_{app}=\frac12w_a.\label{eq:N-mixed}
\end{align}
Here $(G^*)_p$ is its $p$-th row with the barred index raised. Define $Z$ by
\[
{}^tZ=2G+\frac32\Phi-\frac\chi2I,
\qquad L(Z)_{abq}=\frac13\sum_r(\eps_{qbr}Z_{ar}+\eps_{qar}Z_{br}).
\]
Then $\tr Z=0$ and
\begin{equation}\label{eq:hook-norm}
D=S+L(Z),\qquad |D|^2=|S|^2+\frac23|Z|^2.
\end{equation}
In particular, with
\begin{equation}\label{eq:FH-def}
F=\frac12(\Phi+\Phi^*),\quad \Psi=\frac12(\Phi-\Phi^*),\quad
H=\frac12(G+G^*),\quad F_0=F-\frac\chi3I,
\end{equation}
one has
\begin{equation}\label{eq:hook-norm-FH}
|D|^2=|S|^2+\frac32\left|F_0+\frac43H\right|^2+\frac16|\Psi|^2,
\end{equation}
and
\begin{equation}\label{eq:mixed-lower-bound}
|\nabla''\Nt|^2\geq|\Sym_3\nabla''\Nt|^2\geq\frac1{12}|\B|^2|G|^2.
\end{equation}
\end{lemma}
\begin{proof}
We first derive the transport of $Q$. The second Bianchi identity gives us
\begin{equation}\label{eq:exterior-bianchi}
R_{i\bar jk\bar\ell,p}-R_{p\bar jk\bar\ell,i}
=-\sum_tT^t_{pi}R_{t\bar jk\bar\ell}, \ \ \ \ \ \ \ \ \  \forall  \ 1\leq i,j,k,\ell, p\leq 3.
\end{equation}
For example, in a connection-normal unitary frame at a point, $[e_p,e_i]=-T(e_p,e_i)$; this fixes the sign in \eqref{eq:exterior-bianchi}. Fix a formal barred vector $\bar Y$ whose coefficients are covariantly constant at the point, and put $r_{ik}=R(e_i,\bar Y,e_k,\bar Y)$. By \eqref{eq:polarization}, $r_{ik}=-r_{ki}$. The three cyclic instances of \eqref{eq:exterior-bianchi} give
\begin{equation}\label{eq:r-transport}
2r_{ik,p}=\sum_t(T^t_{ip}r_{tk}-T^t_{ki}r_{tp}+T^t_{pk}r_{ti}).
\end{equation}
Set $q_a=\frac12\sum_{i,k}\eps_{ika}r_{ik}$, so that $r_{ik}=\sum_a\eps_{ika}q_a$. Multiplying \eqref{eq:r-transport} by $\eps_{ika}/4$ and summing yields
\[
q_{a,p}=\sum_t\At_{at}r_{tp}-\frac12\delta_{ap}\sum_{k,t}\At_{kt}r_{tk}.
\]
The last contraction is $-\sum_b\eta_bq_b$, whereas the skew part of the first term is
$\frac12(\eta_pq_a-\delta_{ap}\sum_b\eta_bq_b)$. Hence
\[
q_{a,p}=\frac12\eta_pq_a+\sum_{b,t}\eps_{pbt}\B_{at}q_b.
\]
Comparing its quadratic coefficients proves \eqref{eq:Q-transport}; symmetrization gives \eqref{eq:S-transport}.

Next consider the curvature of $\Sym^2\!E\otimes\KM$. Its action on $\B$ is
\[
(R^{\B}_{p\bar q}\B)_{ab}
=\sum_r(R_{p\bar qr\bar a}\B_{rb}+R_{p\bar qr\bar b}\B_{ar})
-\rho_{p\bar q}\B_{ab}.
\]
The last term is the curvature of the canonical factor. By polarization and the definition of $Q$,
\[
\frac12(R_{p\bar cr\bar a}+R_{p\bar ar\bar c})
=\sum_u\eps_{pru}Q_{uac}.
\]
Consequently,
\begin{equation}\label{eq:B-curv-sym}
\Sym_{abc}(R^{\B}_{p\bar c}\B)_{ab}
=-2\Sym_{abc}\sum_{u,v}\eps_{puv}\B_{av}Q_{ubc}
-I_{\B}(\rho_p)_{abc}.
\end{equation}
The mixed Hessian commutator has no torsion term. Differentiating the definition of $\Nt$ gives
\[
\Nt_{abp,\bar q}=D_{abq,p}-(R^{\B}_{p\bar q}\B)_{ab}
-\frac12\Phi_{p\bar q}\B_{ab}-\frac12\eta_pD_{abq}.
\]
After symmetrizing in $a,b,q$, the nonlinear terms cancel by \eqref{eq:S-transport} and \eqref{eq:B-curv-sym}. The remaining expression is
$I_{\B}(\rho_p)-\frac12I_{\B}(\Phi_p)=-\frac12I_{\B}((G^*)_p)$.

The pure torsion Bianchi identity is
\[
\sum_{\mathrm{cyclic}}\bigl((\nabla_XT)(Y,Z)+T(T(X,Y),Z)\bigr)=0.
\]
Taking $(X,Y,Z)=(e_s,e_i,e_k)$, contracting the output with $s$, and summing gives
\[
\sum_sT^s_{ik,s}=W_{ik},\qquad
W_{ik}=\eta_{k,i}-\eta_{i,k}+\sum_rT^r_{ik}\eta_r.
\]
Contract with $\eps_{ika}/2$. If $c_a=\sum_{p,r}\eps_{apr}\eta_{r,p}$, this gives
\[
\sum_p\At_{ap,p}=w_a,\qquad
w_a=c_a+\sum_s\B_{as}\eta_s,\qquad
\sum_p\At_{ap,p}=\sum_p\B_{ap,p}+\frac12c_a.
\]
Combining the last two equations proves the divergence formula in \eqref{eq:N-mixed}.

To compute the remaining part of $D$, set $Z_{rb}=\sum_{a,s}\eps_{bas}D_{sra}$. Its trace is zero. Substitute \eqref{eq:D-from-R}. The contraction of the first summand is
$\frac12(\rho^{(2)}-D)_{b\bar r}=(G+\Phi/2)_{b\bar r}$, where $D$ here denotes the auxiliary curvature contraction in the proof of Lemma~\ref{lem:ricci}. Before multiplication by $-1/2$, the second summand gives
\[
D_{b\bar r}-\rho_{b\bar r}-\rho^{(2)}_{b\bar r}+C_{b\bar r}
+\delta_{br}(\tr\rho-\tr C)
=(4\rho-2\Phi-2\Phi^*+\chi I)_{b\bar r}.
\]
It therefore contributes $(G+\Phi-\chi I/2)_{b\bar r}$. This proves the stated formula for ${}^tZ$.

The tensor $L(Z)$ has zero full symmetrization and has the same alternating-symbol contraction as $D$. Thus $D-S-L(Z)$ is fully symmetric with zero full symmetrization, and is zero. The ordered norm calculation gives
\[
|L(Z)|^2=\frac19\bigl(4|Z|^2+2(|Z|^2-|\tr Z|^2)\bigr)=\frac23|Z|^2.
\]
It is orthogonal to $S$, proving \eqref{eq:hook-norm}. Since $G=H-\Psi$,
$\,{}^tZ=\frac32F_0+2H-\frac12\Psi$. Hermitian and skew-Hermitian matrices are orthogonal for the real part of the Hermitian product, so \eqref{eq:hook-norm-FH} follows.

Finally, direct expansion of the three symmetric summands gives
\[
|I_{\B}(v)|^2=\frac13|\B|^2|v|^2+\frac23|\B\bar v|^2
\geq\frac13|\B|^2|v|^2.
\]
Here the barred vector in the matrix product records the Hermitian contraction of the vector slot; only the displayed lower bound is needed. Apply this to each raised row of $G^*$ in \eqref{eq:N-mixed}. The factor $1/2$ contributes $1/4$ after taking the squared norm. Symmetrization is an orthogonal projection, so it does not increase the norm. This proves \eqref{eq:mixed-lower-bound}.
\end{proof}

\vspace{0.3cm}

\section{The integral identities}\label{sec:integral}

From this section onward, $M$ is compact and connected, without boundary. All integrals are taken with respect to the original volume $dV=\omega^3/6$, unless a measure is displayed explicitly. On real functions put
\begin{equation}\label{eq:operators}
\Delta u=g^{i\bar j}\partial_i\partial_{\bar j}u,\quad
\LL u=\Delta u-\re\ip{\partial u}{\eta},\quad
\PP u=-\LL u+\left(\frac\chi2-\frac{|\eta|^2}{4}\right)u,
\quad a_u=\partial u-\frac u2\eta.
\end{equation}

\begin{lemma}\label{lem:stokes}
For real smooth functions $f,u,\vartheta$ on $M$,
\begin{align}
\int_M f\LL u&=-\re\int_M\ip{\partial u}{\partial f},\label{eq:stokes-L}\\
\int_M f\chi&=\int_M f|\eta|^2-\re\int_M\ip{\eta}{\partial f},\label{eq:stokes-chi}\\
\int_M\vartheta u\PP u&=\int_M\vartheta|a_u|^2
+\re\int_M u\ip{\partial\vartheta}{a_u}.\label{eq:stokes-P}
\end{align}
In particular, $\int_Mu\PP u=\int_M|a_u|^2\geq0$. If $\psi$ is smooth and real, $\PP\psi=F$ with $F\geq0$, and $v=\max\{-\psi,0\}$, then
\begin{equation}\label{eq:negative-part}
0\leq\int_MvF=-\int_M|a_v|^2\leq0.
\end{equation}
\end{lemma}
\begin{proof}
Let $b=\det(g_{i\bar j})$ in a holomorphic coordinate chart, with inverse coefficients defined by $\sum_jg_{k\bar j}g^{i\bar j}=\delta_{ki}$. The torsion trace is
\[
\eta_i=\sum_{k,\ell}g^{k\bar\ell}\partial_k g_{i\bar\ell}-\partial_i\log b.
\]
Differentiating the inverse matrix and using the conjugate formula for $\eta$ gives
\begin{equation}\label{eq:volume-div}
\sum_j\partial_{\bar j}(b g^{i\bar j})=-b\sum_\ell g^{i\bar\ell}\bar\eta_\ell.
\end{equation}
Integration by parts, first for compactly supported functions in the chart, now yields
\[
\int_M f\Delta u=-\int_Mg^{i\bar j}(\partial_i u)(\partial_{\bar j}f)
+\int_Mf g^{i\bar j}(\partial_i u)\bar\eta_j.
\]
Taking real parts proves \eqref{eq:stokes-L}. A partition of unity gives the global formula: the additional derivative terms sum to zero because the partition functions sum to one.

Since the barred Chern coefficients vanish in holomorphic coordinates,
$\chi=g^{i\bar j}\partial_{\bar j}\eta_i$. Apply \eqref{eq:volume-div} once more, now to $f\eta_i$, to obtain \eqref{eq:stokes-chi}. Apply \eqref{eq:stokes-L} with test function $\vartheta u$ and \eqref{eq:stokes-chi} with test function $\vartheta u^2$. Completing the square gives \eqref{eq:stokes-P}, including the term containing $\partial\vartheta$.

The unweighted bilinear form is continuous on real $W^{1,2}(M)$ because all its coefficients are bounded. Local mollification and a finite partition of unity therefore extend the identity from smooth functions to $W^{1,2}$. The negative part $v$ is Lipschitz and satisfies $\partial v=-\partial\psi$ on $\{\psi<0\}$ and $\partial v=0$ almost everywhere on its complement. Thus testing $\PP\psi=F$ against $v$ gives the negative of the square identity for $v$, establishing \eqref{eq:negative-part}, and we have completed the proof of the lemma.
\end{proof}

To streamline our later computations, let us introduce more notation:
\begin{equation}\label{eq:energy-notation}
t=|\B|^2,\quad \KB=\B^*\B,\quad \KB_0=\KB-\frac t3I,
\quad r_G=\re\tr(G^2),\quad
(H_w)_{ab}=\frac12(w_{a,\bar b}+w_{b,\bar a}),
\end{equation}
and the notation in \eqref{eq:FH-def} remains in force.

\begin{proposition}\label{prop:energy}
The primitive norm satisfies
\begin{equation}\label{eq:primitive-norm}
\PP t=-|\Nt|^2-|D|^2-\tr\bigl(\KB(F+H)\bigr)+2\re\ip{H_w}{\B}.
\end{equation}
For every real smooth function $\vartheta$,
\begin{equation}\label{eq:weighted-energy}
\int_M\vartheta\mathcal E+\re\int_Mt\ip{\partial\vartheta}{a_t+2\B^*w}=0,
\end{equation}
where
\begin{equation}\label{eq:energy-density}
\begin{split}
\mathcal E={}&|a_t+\B^*w|^2+t|w|^2-|\B^*w|^2+t|\Nt|^2+t|S|^2\\
&+\frac32t\left|F_0+\frac43H+\frac13\KB_0\right|^2
+\frac t6|H-\KB_0|^2
+\frac{\chi t^2-t|\KB_0|^2}{3}-\frac t6r_G.
\end{split}
\end{equation}
\end{proposition}
\begin{proof}
The mixed equation in Lemma~\ref{lem:differential} can be written as
\begin{equation}\label{eq:Nmixed-expanded}
\Nt_{abp,\bar q}+\Nt_{aqp,\bar b}+\Nt_{bqp,\bar a}
=-\frac12\bigl(\B_{ab}(G^*)_{pq}+\B_{aq}(G^*)_{pb}+\B_{bq}(G^*)_{pa}\bigr).
\end{equation}
Set $q=p$ and sum. Using $\tr G=0$ and $\sum_p\Nt_{app}=w_a/2$ gives
\[
\sum_p\Nt_{abp,\bar p}=-(J_G)_{ab}-(H_w)_{ab},\qquad
J_G=\Sym_2(\B G^*).
\]
If $E^0_{abp}=\B_{ab,p}=\Nt_{abp}+\eta_p\B_{ab}/2$, it follows that
\begin{equation}\label{eq:E0-div}
\sum_pE^0_{abp,\bar p}=\frac\chi2\B_{ab}
+\frac12\sum_p\eta_pD_{abp}-(J_G)_{ab}-(H_w)_{ab}.
\end{equation}
The norm product rule and the mixed Hessian commutator give
\begin{equation}\label{eq:norm-rule}
\Delta t=|E^0|^2+|D|^2
+2\re\ip{\sum_pE^0_{\cdot\cdot p,\bar p}}{\B}
+\ip{\mathcal M\B}{\B},
\end{equation}
where the mean curvature of $\Sym^2E\otimes\KM$ is
\[
(\mathcal M\B)_{ab}=\sum_r\bigl(\rho^{(2)}_{r\bar a}\B_{rb}
+\rho^{(2)}_{r\bar b}\B_{ar}\bigr)-(\tr\rho)\B_{ab}.
\]
The sign in \eqref{eq:norm-rule} follows from commuting $\nabla_p\nabla_{\bar p}$ to $\nabla_{\bar p}\nabla_p$; the mixed torsion is zero. Lemma~\ref{lem:ricci} gives
\[
\rho=\frac12(F-H),\qquad \rho^{(2)}=\frac12F+\frac32H,
\qquad \tr\rho=\frac\chi2.
\]
Consequently,
\[
\ip{\mathcal M\B}{\B}=\tr\bigl(\KB(F+3H)\bigr)-\frac\chi2t,
\qquad \re\ip{J_G}{\B}=\tr(\KB H).
\]
Insert these formulas and \eqref{eq:E0-div} into \eqref{eq:norm-rule}, and subtract the drift term. The first-order terms combine as
$|E^0|^2-\re\ip{E^0}{\eta\otimes\B}=|\Nt|^2-|\eta|^2t/4$.
This proves \eqref{eq:primitive-norm}.

We next integrate the $H_w$ term. The vector divergence rule for the original volume is
\[
\int_M\sum_i(\nabla_iV)^i=\int_M\sum_i\eta_iV^i.
\]
Indeed, the difference between the Chern divergence and the coordinate volume divergence is $\eta_iV^i$. Its conjugate version, applied to the complete tensor contraction, together with
$\sum_a\B_{ab,a}=w_b/2+\sum_a\eta_a\B_{ab}/2$, gives
\begin{equation}\label{eq:Hw-int}
2\re\int_M\vartheta t\ip{H_w}{\B}
=-\int_M\vartheta t|w|^2-2\re\int_M\vartheta\ip{w}{\B a_t}
-2\re\int_Mt\ip{w}{\B\partial\vartheta}.
\end{equation}
All factors are differentiated here, including the weight. For the algebraic completion let us set $T_1=F_0+4H/3$. Formula~\eqref{eq:hook-norm-FH} implies
\[
|D|^2+\tr\bigl(\KB(F+H)\bigr)
\ =\ |S|^2+\frac32|T_1+\KB_0/3|^2+\frac16|\Psi|^2+\frac\chi3t-\frac16|\KB_0|^2-\frac13\tr(\KB_0H).
\]
Since $G=H-\Psi$, one has $r_G=|H|^2-|\Psi|^2$, and thus
\[
\frac16|\Psi|^2-\frac16|\KB_0|^2-\frac13\tr(\KB_0H)
=\frac16|H-\KB_0|^2-\frac13|\KB_0|^2-\frac16r_G.
\]
Multiply \eqref{eq:primitive-norm} by $\vartheta t$ and integrate. Using \eqref{eq:stokes-P} and \eqref{eq:Hw-int}, the first-order terms become
\[
|a_t|^2+2\re\ip{w}{\B a_t}+t|w|^2
=|a_t+\B^*w|^2+t|w|^2-|\B^*w|^2.
\]
The two remaining weight derivatives give the second integral in \eqref{eq:weighted-energy}. This completes  the proof of the proposition.
\end{proof}

\begin{proposition}\label{prop:scalar}
On $M$ one has the scalar identity
\begin{equation}\label{eq:scalar}
\PP(\chi-2t)=3|S|^2+2|\Nt|^2.
\end{equation}
\end{proposition}
\begin{proof}
We give the trace calculation in detail. Set
\[
\begin{gathered}
(T_0)^r_{ba}=\sum_d\eps_{bad}\B_{dr},\qquad
h_2=|H|^2,\qquad a_2=|\Psi|^2,\\
f_H=\tr(FH),\qquad k_H=\tr(\KB H),\qquad c_F=\tr(\KB F),\qquad b_w=\re\ip{H_w}{\B}.
\end{gathered}
\]
Note that all these scalar quantities are real. We will carry out the proof in the following steps.

\vspace{0.15cm}

\step{Step 1. The trace-gradient identity.}
Let $\mathcal U_{a\ell j}=\At_{a\ell,\bar j}$ and
$\mathcal H_{ab}=\sum_{j,\ell}\eps_{j\ell b}\,\mathcal U_{a\ell j}$.
The formula for ${}^t\!Z$ in Lemma~\ref{lem:differential} gives
\begin{equation}\label{eq:sc-H}
{}^t\!\mathcal H={}^t\!Z+\frac12\Phi-\frac\chi2I
=2G+2\Phi-\chi I=-4\rho+2\Phi+2\Phi^*-\chi I.
\end{equation}
Here the additional contraction is
$\frac12\sum_{j,\ell,k}\eps_{j\ell b}\eps_{a\ell k}\Phi_{kj}
=(\Phi_{ba}-\chi\delta_{ab})/2$. The Chern connection has no $(0,2)$ curvature. Its full Hessian commutator therefore gives
\[
\mathcal U_{a\ell j,\bar m}-\mathcal U_{a\ell m,\bar j}
=\sum_r\overline{T^r_{jm}}\,\mathcal U_{a\ell r}.
\]
Use this with $(\ell,j,m)=(1,2,3),(2,3,1),(3,1,2)$ and sum. The three components of $\mathcal H$ are the corresponding alternating differences of $\mathcal U$, whence
\begin{equation}\label{eq:sc-divH}
\sum_j\mathcal H_{aj,\bar j}=-\sum_{\ell,r}\overline{\At_{\ell r}}\,\mathcal U_{a\ell r}.
\end{equation}
The same commutator on $\eta$, and closedness of the first Chern--Ricci form, give respectively
\begin{align}
\sum_j\Phi_{ja,\bar j}-\chi_{,\bar a}
&=-\sum_{j,r}\overline{T^r_{ja}}\Phi_{jr},\label{eq:sc-divPhi}\\
\sum_j\rho_{ja,\bar j}
&=\frac12\chi_{,\bar a}-\sum_{j,r}\overline{T^r_{ja}}\rho_{jr}.
\label{eq:sc-divrho}
\end{align}
For the latter, one may use $\rho_{i\bar j}=-\partial_i\partial_{\bar j}\log\det(g_{k\bar\ell})$ in a holomorphic chart and then restore the torsion term. Differentiating \eqref{eq:sc-H} and substituting \eqref{eq:sc-divH}--\eqref{eq:sc-divrho} yields
\begin{equation}\label{eq:sc-grad0}
2\sum_j\overline\Phi_{aj,\bar j}-\chi_{,\bar a}+C_a=0,
\qquad C_a=\sum_{\ell,r}\overline{\At_{\ell r}}\,\mathcal U_{a\ell r}
+\sum_{b,r}\overline{T^r_{ba}}(4\rho-2\Phi)_{br}.
\end{equation}

\vspace{0.1cm}

\step{Step 2. Expansion of the source.}
For a matrix $V$ with the slots of $\Phi$, define
\[
X(V)_a=\sum_{b,r}\overline{(T_0)^r_{ba}}V_{br},\qquad
\nu_a=\sum_{\ell,r}\overline{\B_{\ell r}}S_{a\ell r},\qquad
(dG)_a=\sum_bG_{ba,\bar b}.
\]
These are $(0,1)$-forms, or equivalently raised $(1,0)$ vectors. By \eqref{eq:hook-norm},
\[
\mathcal U_{a\ell r}=S_{a\ell r}+L(Z)_{a\ell r}
+\frac12\sum_k\eps_{a\ell k}\Phi_{kr}.
\]
When contracted against $\overline{\B_{\ell r}}$, the $\eps_{r\ell b}$ term vanishes. The remaining $\Phi$ terms cancel after interchanging $\ell,r$. Thus
\begin{equation}\label{eq:sc-primitive-source}
\sum_{\ell,r}\overline{\B_{\ell r}}\mathcal U_{a\ell r}
=\nu_a-\frac23X(G)_a.
\end{equation}
For the skew part of $\At$, formula~\eqref{eq:sc-H} gives
\[
\sum_{\ell,r}\overline{(\At-\B)_{\ell r}}\mathcal U_{a\ell r}
=-\sum_k(G_{ka}+\Phi_{ka})\bar\eta_k+\frac\chi2\bar\eta_a.
\]
The remaining part of $C$ is the full torsion contracted against $-2G+2(\Phi^*-\Phi)$. Its trace part is
$\sum_b\bar\eta_b[G_{ba}-(\Phi^*)_{ba}+\Phi_{ba}]$.
Consequently,
\begin{equation}\label{eq:sc-C}
C=\nu-\frac83X(G)+2X(\Phi^*-\Phi)-\bar\Phi\bar\eta+\frac\chi2\bar\eta.
\end{equation}
The bar on $\Phi$ in this formula is entrywise conjugation, not adjoint. Since $\Phi^*=2\rho+G$, equations \eqref{eq:sc-divrho} and \eqref{eq:sc-grad0} now give
\begin{equation}\label{eq:sc-gradient}
\bp\chi=\frac\chi2\bar\eta-\nu+2X(\Phi)+R_0,
\qquad R_0=\frac23X(G)+V_0,
\qquad V_0={}^tG\bar\eta-2dG.
\end{equation}

\vspace{0.1cm}

\step{Step 3. Curls and divergences.}
For a coefficient matrix $V$ and a raised vector $v$, put
\[
(\Curl V)_{ab;r}=V_{br,a}-V_{ar,b}+\sum_sT^s_{ab}V_{sr},
\qquad D_\eta v=\re\sum_a\left(v_{a,a}-\frac12\eta_av_a\right).
\]
In holomorphic coordinates, $\Phi_{i\bar j}=\partial_{\bar j}\eta_i$ and
$(\Phi^*)_{i\bar j}=\partial_i\bar\eta_j$. Since the first Ricci form is closed,
\begin{equation}\label{eq:sc-curls}
\Curl G=0,\qquad (\Curl G^*)_{ab;r}=W_{ab,\bar r},\qquad
\Curl\Phi=\Curl G^*.
\end{equation}
Contracting the second equation with $\overline{(T_0)^r_{ba}}$ gives
\begin{equation}\label{eq:sc-curlW}
\re\sum_{a,b,r}\overline{(T_0)^r_{ba}}(\Curl G^*)_{ab;r}=-2b_w.
\end{equation}
Indeed, $\sum_{a,b}\eps_{bad}\eps_{abs}=-2\delta_{ds}$, and symmetry of $\B$ replaces the derivative of $w$ by $H_w$.

Differentiate $X(V)$ and antisymmetrize the derivative of $V$ in $a,b$. The derivative of the torsion factor contributes $\overline{Z_{rb}}$. Restoring the torsion term in $\Curl V$ contributes $\tr(\KB V)$ and $\frac12\sum_a\eta_aX(V)_a$; the latter cancels the shift in $D_\eta$. Thus
\begin{equation}\label{eq:sc-divX}
D_\eta X(V)=\re\left\{\sum_{b,r}\overline{Z_{rb}}V_{br}
+\tr(\KB V)+\frac12\sum_{a,b,r}\overline{T_0{}^r{}_{ba}}(\Curl V)_{ab;r}\right\}.
\end{equation}
Using ${}^tZ=2G+3\Phi/2-\chi I/2$ and \eqref{eq:sc-curls}--\eqref{eq:sc-curlW}, we obtain
\begin{equation}\label{eq:sc-divX-values}
\begin{aligned}
D_\eta X(\Phi)&=\frac32|F|^2-\frac12\chi^2+2f_H-\frac12a_2+c_F-b_w,\\
D_\eta X(G)&=2h_2+\frac12a_2+\frac32f_H+k_H,\\
D_\eta X(G^*)&=2h_2-\frac12a_2+\frac32f_H+k_H-b_w.
\end{aligned}
\end{equation}
The relevant pairings are $\re\ip\Phi G=f_H-a_2$,
$\re\ip{G^*}\Phi=f_H+a_2$, and $\re\ip{G^*}G=h_2-a_2$.

Contracting $\Curl G=0$ and using $\tr G=0$ gives
\begin{equation}\label{eq:sc-J}
J_b:=\sum_aG_{ba,a}=-\sum_{a,s}T^s_{ab}G_{sa}
=-\overline{X(G^*)_b}+\frac12(G\eta)_b.
\end{equation}
The curvature action on all slots of $G$ is
\[
(R^G_{p\bar q}G)_{ra}=\sum_sR_{p\bar qs\bar a}G_{rs}
-\sum_sR_{p\bar qr\bar s}G_{sa}.
\]
Using the two auxiliary Ricci contractions from Lemma~\ref{lem:ricci}, it follows that
\begin{equation}\label{eq:sc-curvG}
\re\sum_{a,b}(R^G_{a\bar b}G)_{ba}
=\re\tr\bigl((\Phi^*-\Phi)G\bigr)=-2a_2.
\end{equation}
Here $\Phi^*-\Phi=-2\Psi$, $G=H-\Psi$, and $\tr(\Psi^2)=-a_2$. Let $\divg'v=\sum_av_{a,a}$, and use $\divg''$ for the conjugate-slot contraction. Commuting the two full derivatives in \eqref{eq:sc-J}, with \eqref{eq:sc-curvG}, gives
\begin{equation}\label{eq:sc-divdG}
\re\divg'(dG)=-\re\divg'X(G^*)
+\frac12\re\divg''(G\eta)-2a_2.
\end{equation}
The product rules also give
\begin{equation}\label{eq:sc-products}
\begin{aligned}
& D_\eta({}^tG\bar\eta)=f_H-a_2-\re\sum_a\eta_aX(G^*)_a,\\
& \re\divg''(G\eta)=\re\sum_a\eta_a(dG)_a+f_H+a_2.
\end{aligned}
\end{equation}
Substitution into \eqref{eq:sc-divdG} cancels the $\eta\cdot dG$ terms and yields
\begin{equation}\label{eq:sc-divR0}
\begin{aligned}
D_\eta V_0&=2D_\eta X(G^*)+2a_2=4h_2+a_2+3f_H+2k_H-2b_w,\\
D_\eta R_0&=\frac{16}{3}h_2+\frac43a_2+4f_H+\frac83k_H-2b_w.
\end{aligned}
\end{equation}

\vspace{0.1cm}

\step{Step 4. The scalar cancellation.}
By the product rule and \eqref{eq:S-transport},
\[
\nu_{a,p}=\sum_{\ell,r}\overline{D_{\ell rp}}S_{a\ell r}
+\frac12\eta_p\nu_a
-2\sum_{\ell,r}\overline{\B_{\ell r}}\Sym_{a\ell r}
\sum_{u,v}\eps_{puv}\B_{av}Q_{u\ell r}.
\]
Contract $p=a$ and take real parts. The first term is $|S|^2$. In the last term, the contraction of $\eps$ with the symmetric $\B_{av}$ vanishes; the other two terms agree and give
\[
-\frac43\sum_{a,\ell,r,u,v}\overline{\B_{\ell r}}\eps_{auv}\B_{v\ell}Q_{uar}
=\frac23\sum_{r,v}\KB_{rv}G_{vr},
\]
since $\sum_{a,u}\eps_{auv}Q_{uar}=-G_{vr}/2$. Hence
\begin{equation}\label{eq:sc-divnu}
\re\divg'\nu=|S|^2+\frac12\re\sum_a\eta_a\nu_a+\frac23k_H.
\end{equation}
Take the holomorphic divergence of \eqref{eq:sc-gradient}, its real part, and subtract the drift. Equations \eqref{eq:sc-divX-values} and \eqref{eq:sc-divnu} give
\[
\LL\chi=3|F|^2-\frac12\chi^2-\frac14\chi|\eta|^2-|S|^2
+2c_F+4f_H-a_2-\frac23k_H-2b_w+D_\eta R_0.
\]
For clarity, differentiating $(\chi/2)\bar\eta$ contributes $\chi^2/2$ and half the drift. The other half of the drift is eliminated with \eqref{eq:sc-gradient}; this produces $-\chi|\eta|^2/4$ and the shifted divergence of $R_0$. Using \eqref{eq:sc-divR0}, we obtain
\begin{equation}\label{eq:sc-Lchi}
\begin{split}
\LL\chi={}&3|F|^2-\frac12\chi^2-\frac14\chi|\eta|^2-|S|^2
+2c_F+8f_H+\frac{16}{3}h_2+\frac13a_2+2k_H-4b_w.
\end{split}
\end{equation}
On the other hand, expanding \eqref{eq:primitive-norm} with \eqref{eq:hook-norm-FH} gives
\begin{equation}\label{eq:sc-Lt}
\begin{split}
\LL t={}&|\Nt|^2+|S|^2+\frac32|F|^2-\frac12\chi^2
+4f_H+\frac83h_2+\frac16a_2\\
&+\left(\frac\chi2-\frac{|\eta|^2}{4}\right)t+c_F+k_H-2b_w.
\end{split}
\end{equation}
Subtract twice \eqref{eq:sc-Lt} from \eqref{eq:sc-Lchi}. All six mixed quantities $c_F,k_H,f_H,h_2,a_2,b_w$ cancel, and the remaining equation is
\[
\LL(\chi-2t)=\left(\frac\chi2-\frac{|\eta|^2}{4}\right)(\chi-2t)
-3|S|^2-2|\Nt|^2.
\]
This is just \eqref{eq:scalar}, and we have completed the proof of the proposition.
\end{proof}

\vspace{0.1cm}

\begin{proposition}\label{prop:compact-estimate}
Suppose $3|S|^2+2|\Nt|^2$ is not identically zero. Then
\begin{equation}\label{eq:scalar-gap}
\psi:=\chi-2t\geq0
\end{equation}
on $M$, and
\begin{equation}\label{eq:compact-estimate}
\int_Mt r_G\geq6\int_Mt(|\Nt|^2+|S|^2)
+\frac83\int_Mt^3+2\int_M\psi t^2.
\end{equation}
If $\rank Q\leq1$ everywhere, then $Q=0$.
\end{proposition}
\begin{proof}
Put $F=3|S|^2+2|\Nt|^2$. Apply Lemma~\ref{lem:stokes} to $\PP\psi=F$ and $v=\max\{-\psi,0\}$. Equation~\eqref{eq:negative-part} gives $a_v=0$ almost everywhere and $\int_MvF=0$. On $U=\{\psi<0\}$, $v$ is smooth and positive, so
\[
dv=\frac v2(\eta+\bar\eta).
\]
If $U$ is nonempty, start a piecewise smooth path at a point $p\in U$. Before its first exit from $U$, this equation gives
\[
v(\gamma(s))=v(p)\exp\left(\frac12\int_0^s(\eta+\bar\eta)(\dot\gamma(r))\,dr\right).
\]
The exponent has a finite limit at a first exit, whereas $v$ must tend to zero there. Thus no first exit exists. Connectedness gives $v>0$ everywhere, contradicting $\int_MvF=0$ because $F$ is nonnegative and not identically zero. This proves \eqref{eq:scalar-gap}; no density assumption on $\{F>0\}$ is needed.

Take $\vartheta=1$ in \eqref{eq:weighted-energy}. The elementary bounds
\[
|\B^*w|^2\leq t|w|^2,\qquad
|\KB_0|^2=\sum_i\lambda_i^2-\frac13t^2\leq\frac23t^2,
\]
where $\lambda_i\geq0$ are the eigenvalues of $\KB$ and $\sum_i\lambda_i=t$, give
\[
\mathcal E\geq t(|\Nt|^2+|S|^2)+\frac49t^3+\frac13\psi t^2-\frac16t r_G.
\]
Integrating and multiplying by six proves \eqref{eq:compact-estimate}.

Finally, suppose $\rank Q\leq1$ everywhere and $Q\not\equiv0$. Lemma~\ref{lem:algebra} gives $r_G\leq3|S|^2$ and $S\ne0$ at a point where $Q\ne0$. Thus the preceding inequality applies and gives
\[
0\geq6\int_Mt|\Nt|^2+3\int_Mt|S|^2+\frac83\int_Mt^3+2\int_M\psi t^2\geq0.
\]
It follows that $t=0$ everywhere, and hence $\B=S=0$. This contradicts the choice of the point. Therefore $Q=0$, and we have completed the proof of the proposition.
\end{proof}

\vspace{0.3cm}

\section{Curvature rank and the case \texorpdfstring{$Q=0$}{Q = 0}}\label{sec:rank}
In this section we will treat the case in which this tensor $Q$ vanishes. Note that this step is necessary as $Q$ contains only part of the information of the Chern curvature $R$, so $Q=0$ does not mean $R=0$ yet. For later use, let us write
\begin{equation}\label{eq:C-def}
T_0(X,Y)=T(X,Y)-\frac12\{\eta(Y)X-\eta(X)Y\},\qquad
C_X=T_0(X,\cdot)+\frac12\eta(X)I.
\end{equation}
Thus \eqref{eq:Q-transport} is $\nabla'Q=CQ$, where $C$ acts on the output factor.

\begin{lemma}\label{lem:continuation}
The maximal-rank locus of $Q$ is open and dense. More generally, if $\Lambda^kQ$ vanishes on a nonempty open subset of $M$, then it vanishes identically.
\end{lemma}
\begin{proof}
We first recall an elementary uniqueness argument for a smooth system in one complex variable. Suppose $\partial_{\bar z}v=A(z)v$ on a disk, with $A$ smooth. On a sufficiently small disk, the equation
\[
P=I-\mathcal T(PA),\qquad
\mathcal T f(z)=\frac1\pi\int\frac{f(\xi)}{z-\xi}\,dA(\xi),
\]
has a solution close to $I$. Indeed, $\partial_{\bar z}\mathcal T f=f$, and on a disk of radius $r$,
\[
\sup_z\int\frac{dA(\xi)}{|z-\xi|}\leq4\pi r.
\]
The map on the right is therefore a contraction if $4r\sup|A|<1$. By decreasing $r$, we also have $|P-I|<1/2$, so $P$ is invertible. The solution is smooth in the interior. To see the regularity directly, splitting the difference of the kernels at distance $2\delta$ gives the modulus $C\|f\|_\infty\delta(1+|\log\delta|)$ for $\mathcal T f$. Thus $P$ is locally H\"older. Subtracting the value of the H\"older numerator at the singularity permits one differentiation, and iteration gives smoothness. Consequently
\[
\partial_{\bar z}(Pv)=-PAv+PAv=0.
\]
The identity principle for holomorphic functions, followed through overlapping disks, proves open-set uniqueness for $v$.

Now consider a matrix $A_0$ on a polydisk satisfying $\bp A_0=A_1A_0$, with $A_1$ a smooth matrix-valued $(0,1)$-form. Each column of $\Lambda^k A_0$ satisfies the induced system. If it vanishes on a product of small disks, the preceding argument propagates the vanishing first in one coordinate, then in the next, and finally throughout the polydisk. Overlapping coordinate neighborhoods give the same conclusion on a connected manifold.

Apply this observation to the conjugate of $\nabla'Q=CQ$. Its source is the holomorphic bundle $\Sym^2 E$. In a holomorphic source frame and a smooth output frame, the conjugate equation has precisely the form $\bp A_0=A_1A_0$; all connection coefficients occur on the output side. Hence the observation applies to every exterior power. If $r$ is the maximum rank of $Q$, the set $\{\Lambda^rQ\ne0\}$ is nonempty and open, and its complement has no interior. This completes the proof of the lemma.
\end{proof}

\begin{lemma}\label{lem:pure-trace}
If $\B=0$ on $M$, then $R=0$.
\end{lemma}
\begin{proof}
By \eqref{eq:torsion-split}, $T^l_{ik}=(\delta_{il}\eta_k-\delta_{kl}\eta_i)/2$. Put $q=-\eta/2$ and $\vartheta=q+\bar q$. In holomorphic coordinates this says
\[
\partial_i g_{k\bar l}-\partial_k g_{i\bar l}
=q_i g_{k\bar l}-q_k g_{i\bar l},
\]
so $\partial\omega=q\wedge\omega$ and $d\omega=\vartheta\wedge\omega$. Thus $d\vartheta\wedge\omega=0$. Multiplication by $\omega$ is injective on real two-forms in real dimension six. For example, if $L\alpha=\omega\wedge\alpha$ and $\Lambda$ is its symplectic contraction, then $[\Lambda,L]\alpha=(3-k)\alpha$ on $k$-forms. If $k=2$ and $L\alpha=0$, we obtain $\alpha=-L\Lambda\alpha$, and then $\Lambda\alpha=-3\Lambda\alpha$, so $\alpha=0$. It follows that $d\vartheta=0$.

Locally write $\vartheta=df$ with $f$ real. The differences of these local functions are constant on overlaps; hence
\[
\mathfrak a=\sqrt{-1}\,\partial\bp f,\qquad
A_{i\bar j}=\partial_i\partial_{\bar j}f,\qquad a=\tr_g\mathfrak a
\]
are globally defined. Since $\eta=-2\partial f$, the mixed torsion identity gives
\[
T^l_{ik,\bar j}=\delta_{kl}A_{i\bar j}-\delta_{il}A_{k\bar j}.
\]
Together with curvature reality and \eqref{eq:polarization}, this determines the curvature as
\begin{equation}\label{eq:pure-trace-R}
R_{i\bar jk\bar l}
=-\frac34A_{i\bar j}\delta_{kl}
+\frac14A_{k\bar j}\delta_{il}
+\frac14A_{i\bar l}\delta_{kj}
+\frac14A_{k\bar l}\delta_{ij}.
\end{equation}
Indeed, the right side has the required reality, its holomorphic sectional curvature is zero, and its difference under interchange of $i,k$ is $-A_{i\bar j}\delta_{kl}+A_{k\bar j}\delta_{il}$. Subtracting it from $R$ leaves a tensor symmetric in both the unbarred and barred pairs. Its diagonal polynomial vanishes, so polarization makes the difference zero.

Taking the first Ricci contraction in \eqref{eq:pure-trace-R}, we obtain
\[
\rho^{\rm form}=\frac14a\omega-\frac74\mathfrak a.
\]
Both $\rho^{\rm form}$ and $\mathfrak a$ are closed. Thus $(da+a\vartheta)\wedge\omega=0$, and injectivity on one-forms gives $da+a\vartheta=0$. Along paths this is a scalar linear equation, so $a$ is either identically zero or nowhere zero.

In the first case, $\chi=-2a=0$, and \eqref{eq:stokes-chi} with weight one gives $\eta=0$. Equation~\eqref{eq:pure-trace-R} then gives $R=0$. In the second case, $f=-\log|a|$ is a global real function satisfying $df=\vartheta$. The metric $g_0=e^{-f}g$ is K\"ahler, and $ae^f=c_0$ is a nonzero real constant. Since $a=\Delta_g f=e^{-f}\Delta_{g_0}f$, we obtain $\Delta_{g_0}f=c_0$. Integration against the K\"ahler volume gives $0=c_0\operatorname{Vol}_{g_0}(M)$, a contradiction. This completes the proof.
\end{proof}

\begin{proposition}\label{prop:rank}
The maximum output rank of $Q$ is either zero or two.
\end{proposition}
\begin{proof}
The rank-one case is excluded by Proposition~\ref{prop:compact-estimate}. Suppose the maximum rank is three. Differentiating $\nabla'Q=CQ$ and using the vanishing of the pure Chern curvature gives
\begin{equation}\label{eq:J-Q}
J_{pq}Q=0,\qquad
J_{pq}=(\nabla_p C)_q-(\nabla_q C)_p+C_{T(e_p,e_q)}-[C_p,C_q].
\end{equation}
Lemma~\ref{lem:continuation} gives a dense rank-three locus. There $J=0$, and continuity gives $J=0$ on $M$. Since $\tr C=3\eta/2$, we have $\tr J=3\partial\eta/2$, and therefore $W=0$.

To use the full equation $J=0$, put $L_p=T(e_p,\cdot)$ and $H_p=-e_p\otimes\eta/2+\eta_pI$, so $C=L+H$. For an endomorphism-valued one-form,
\[
(d^\nabla L)_{pq}=(\nabla_pL)_q-(\nabla_qL)_p+L_{T(e_p,e_q)}.
\]
The pure torsion Bianchi identity gives
\begin{equation}\label{eq:L-compat}
(d^\nabla L)_{pq}-[L_p,L_q]=-(\nabla_{\cdot}T)(e_p,e_q).
\end{equation}
In components, the quadratic terms are $T^s_{pq}T^a_{sb}-T^a_{ps}T^s_{qb}+T^a_{qs}T^s_{pb}$, exactly the three cyclic torsion products. The remaining terms are
\begin{align*}
(d^\nabla H)_{pq}
&=-\tfrac12e_q\otimes\nabla_p\eta+\tfrac12e_p\otimes\nabla_q\eta
-\tfrac12T(e_p,e_q)\otimes\eta+W_{pq}I,\\
[L_p,H_q]+[H_p,L_q]
&=-T(e_p,e_q)\otimes\eta+\tfrac12e_q\otimes(\eta\circ L_p)
-\tfrac12e_p\otimes(\eta\circ L_q),\\
[H_p,H_q]&=\tfrac14(\eta_qe_p\otimes\eta-\eta_pe_q\otimes\eta).
\end{align*}
Substitution in $J=0$, with $W=0$, yields
\[
(\nabla_bT)^a_{pq}=\frac12\eta_bT^a_{pq}
+\frac12(\delta_{pa}U_{qb}-\delta_{qa}U_{pb}),\qquad
U_{qb}=\eta_{b,q}+\sum_s\eta_sT^s_{qb}-\frac12\eta_q\eta_b.
\]
Contract with $\eps_{pql}/2$ and symmetrize in $a,l$. The $U$-term is skew in these indices, and hence
\begin{equation}\label{eq:fullrank-recurrent}
\nabla'_b\B=\frac12\eta_b\B,\qquad \Nt=0.
\end{equation}

If $\B=0$ identically, Lemma~\ref{lem:pure-trace} contradicts the assumed rank of $Q$. Otherwise the conjugate metric dual of \eqref{eq:fullrank-recurrent} is a scalar-coefficient barred system on the holomorphic dual of $\Sym^2 E\otimes\KM$. The proof of Lemma~\ref{lem:continuation} shows that $\{\B\ne0\}$ is dense. By \eqref{eq:mixed-lower-bound}, $G=0$ on this set, and hence on $M$. Now \eqref{eq:Q-decomp} gives $Q=-S/2$, so $S$ is nonzero somewhere. Apply \eqref{eq:compact-estimate}. Its left side is zero, while its right side contains $\frac83\int_Mt^3$ and only nonnegative terms. Thus $\B=0$, again a contradiction. The rank-three case is impossible.
\end{proof}

We will use the following positive density twice. Its construction does not change the metric.
\begin{lemma}\label{lem:adjoint-density}
There is a smooth positive function $v$ on $M$ such that
\begin{equation}\label{eq:adjoint-density}
\int_Mv\,dV=1,\qquad \int_Mv\Delta u\,dV=0
\end{equation}
for every smooth function $u$. If $f=\frac12\log v$ and $\eta_f=\eta-2\partial f$, then
\begin{equation}\label{eq:weighted-trace}
\int_Mv\chi\,dV=\int_Mv|\eta_f|^2\,dV.
\end{equation}
Moreover, for $\Omega=v\,\omega^2/2$,
\begin{equation}\label{eq:Omega-weight}
\partial\Omega=-\eta_f\wedge\Omega,\qquad \partial\bp\Omega=0.
\end{equation}
\end{lemma}
\begin{proof}
On the real subspace $\{\Delta u+c:u\in C^\infty(M,\R),\ c\in\R\}$ of $C^0(M,\R)$, define $\ell(\Delta u+c)=c$. Evaluation at a maximum and a minimum of $u$ proves both that this is well defined and that
\[
\min(\Delta u+c)\leq c\leq\max(\Delta u+c).
\]
Thus $\ell$ has norm one and $\ell(1)=1$. Extend it by the real Hahn--Banach theorem without increasing its norm. The extension is positive: if $b\geq0$ and $B_0=\max b$, then $\ell(B_0-b)\leq\|B_0-b\|_\infty\leq B_0$, so $\ell(b)\geq0$. The Riesz representation theorem gives a probability measure $\mu$ with $\int_M\Delta u\,d\mu=0$. Distributional elliptic regularity gives $d\mu=v\,dV$ with $v$ smooth and nonnegative.

For completeness, $v$ is strictly positive even though the zeroth-order coefficient of $\Delta^*$ need not have a sign. In a real coordinate ball about a putative zero, write $\Delta^*=a^{ij}\partial_i\partial_j+b^i\partial_i+d$, where $(a^{ij})$ is positive definite. Choose $A>0$ large and then shrink the ball so that $p=e^{-A|x|^2}$ satisfies
\[
p^{-1}\Delta^*p=-2A\tr a+4A^2a^{ij}x_ix_j-2Ab^ix_i+d<0.
\]
The conjugated operator $p^{-1}\Delta^*(p\,\cdot)$ has nonpositive zeroth-order term. The strong maximum principle applied to $-v/p$ makes it zero near an interior zero. Thus the zero set of $v$ is open and closed, contrary to connectedness and $\int_Mv=1$. This proves \eqref{eq:adjoint-density}.

Apply the Stokes identities to the weight $v=e^{2f}$. They give
\[
\int_Mv\chi=\int_Mv\bigl(|\eta|^2-2\re\ip{\eta}{\partial f}\bigr),\qquad
0=\int_Mv\Delta f=\int_Mv\re\ip{\partial f}{\eta}-2\int_Mv|\partial f|^2.
\]
Their combination is \eqref{eq:weighted-trace}. The torsion trace identity $\partial(\omega^2/2)=-\eta\wedge\omega^2/2$ gives the first formula in \eqref{eq:Omega-weight}. Finally,
\[
\sqrt{-1}\,\partial\bp u\wedge\omega^2/2=(\Delta u)dV.
\]
Multiplying by $v$, integrating, and using \eqref{eq:adjoint-density} proves $\partial\bp\Omega=0$ distributionally, hence smoothly.
\end{proof}

\begin{proposition}\label{prop:Qzero}
If $Q=0$ identically, then $R=0$ identically.
\end{proposition}
\begin{proof}
By Lemma~\ref{lem:algebra}, $G=0$, so $\Phi=2\rho$ and $\chi=2s$, where $s=\tr\rho$. The equation $Q=0$ makes $R$ skew in its barred slots, and curvature reality makes it skew in its unbarred slots as well. In particular $\rho^{(2)}=\rho$ in this case.

\step{The skew curvature and its image.}
Let $F=\Lambda^2E^*\simeq E\otimes\KM$, with the exterior metric. In a unitary frame and its matching volume $\nu$, the forms $\iota_{e_a}\nu$ are unitary. There is a Hermitian coefficient matrix $K^R$ such that
\begin{equation}\label{eq:skew-curvature}
R_{i\bar jk\bar l}=\sum_{a,b}\eps_{ika}\eps_{jlb}K^R_{a\bar b},\qquad
\rho_{i\bar j}=k\delta_{ij}-K^R_{j\bar i},\qquad s=2k,
\end{equation}
where $k=\tr K^R$. The tensor represented by $K^R$ is in $F\otimes\bar F$, and $|R|^2=4|K^R|^2$.

The Bianchi argument used in \eqref{eq:r-transport} applies here with any fixed barred pair, not merely two equal barred arguments. Since the unbarred pair is skew, contraction with $\eps_{ika}/4$ gives
\begin{equation}\label{eq:skew-transport}
\nabla'_pK^R=C_pK^R,\qquad \nabla''_pK^R=K^RC_p^*.
\end{equation}
The second equality follows by conjugation. In a Chern-parallel unitary frame along a real path, these equations read $(K^R)'=CK^R+K^RC^*$. If $P'=CP$ and $P(0)=I$, then
\[
K^R(t)=P(t)K^R(0)P(t)^*.
\]
Thus both rank and inertia are constant. If the rank is zero there is nothing to prove. Suppose it is $r\in\{1,2,3\}$.

Define $D'=\nabla'-C$ and $D''=\bp_F$. Equation~\eqref{eq:skew-transport} says $DK^R=0$. Its constant-rank image $\mathscr S$ is therefore $D$-preserved and, in particular, $\bp$-preserved. It is a holomorphic subbundle: in a holomorphic trivialization, a smooth graph frame has the form $(I,A)^t$; its barred derivative $(0,\bp A)^t$ can lie in the same graph only if $\bp A=0$.

\step{A flat determinant line.}
The mixed torsion identity and $\Phi=2\rho$ give
\[
(\nabla_{\bar j}C_p)^a_i=-2R_{p\bar ji\bar a}-\delta_{pa}\rho_{i\bar j}+2\rho_{p\bar j}\delta_{ai}.
\]
The curvature of $F=E\otimes\KM$ is $R^E-\rho I$. Changing its $(1,0)$ connection by $-C$ consequently gives
\begin{equation}\label{eq:D-curvature}
(R^D)_{p\bar j}{}^a_i
=-R_{p\bar ji\bar a}+\rho_{p\bar j}\delta_{ai}-\delta_{pa}\rho_{i\bar j}
=\delta_{pj}K^R_{a\bar i}-\delta_{ij}K^R_{a\bar p}.
\end{equation}
For the last equality, expand the two alternating symbols in \eqref{eq:skew-curvature}:
\begin{align*}
R_{p\bar ji\bar a}
={}&k(\delta_{pj}\delta_{ia}-\delta_{pa}\delta_{ij})
-\delta_{pj}K^R_{a\bar i}-\delta_{ia}K^R_{j\bar p}\\
&+\delta_{pa}K^R_{j\bar i}+\delta_{ij}K^R_{a\bar p}.
\end{align*}
Every endomorphism in \eqref{eq:D-curvature} has image in $\mathscr S$ and trace $\rho_{p\bar j}$. Since $D$ preserves $\mathscr S$, its quotient has zero mixed curvature. The curvature trace on $\mathscr S$ is therefore $\rho$.

The nondegenerate Hermitian tensor induced by $K^R$ on $\mathscr S$ is $D$-parallel. Its inverse is a possibly indefinite Hermitian metric. The absolute value of its determinant gives a positive metric on $\det\mathscr S$, preserved by the induced connection, whose barred part is the ordinary holomorphic structure. This connection is the Chern connection of that positive line metric. Its curvature is of type $(1,1)$ and has real form $\sqrt{-1}\,\rho_{p\bar j}\varphi_p\wedge\bar\varphi_j$. It follows that the line bundle
\[
\mathscr L=\det\mathscr S\otimes\KM
\]
has a flat Hermitian metric.

\step{Determinant sections.}
The determinant of the inclusion $\mathscr S\longrightarrow F$ is a nowhere-zero holomorphic section
\[
\sigma\in H^0\bigl(M,\Lambda^rF\otimes\KM\otimes\mathscr L^*\bigr).
\]
For $r=1,2,3$, these bundles are, respectively,
\[
E\otimes\KM^2\otimes\mathscr L^*,\qquad
E^*\otimes\KM^2\otimes\mathscr L^*,\qquad
\KM^3\otimes\mathscr L^*.
\]
Equip them with the induced metric and the flat metric on $\mathscr L$. The curvature mean acts on $\sigma$ by $-a_rs$, where
\[
a_1=2,\qquad a_2=\frac52,\qquad a_3=3.
\]
Indeed, for $r=1$ the vector factor lies in the rank-one image of $K^R$, on which the Ricci endomorphism $kI-K^R$ vanishes. The canonical factor contributes $-2s$. For $r=2$, the covector factor annihilates that image, so the dual tangent factor contributes $-k=-s/2$, in addition to $-2s$. For $r=3$, only the canonical power contributes.

Put $u=|\sigma|^2>0$. The Chern product rule and Cauchy--Schwarz give
\begin{equation}\label{eq:det-section-bound}
\Delta\log u
=\frac{|\nabla'\sigma|^2}{u}-\frac{|\partial u|^2}{u^2}
-\frac{\ip{\operatorname{Mean}\sigma}{\sigma}}{u}
\geq a_rs.
\end{equation}
Use the positive density $v$ in Lemma~\ref{lem:adjoint-density}. Since $s=\chi/2$, integration of \eqref{eq:det-section-bound} yields
\[
0\geq\frac{a_r}{2}\int_Mv\chi
=\frac{a_r}{2}\int_Mv|\eta-2\partial f|^2\geq0,
\qquad f=\frac12\log v.
\]
Hence $\eta=2\partial f$, and $\Phi=2\rho$ now gives the global formula
\begin{equation}\label{eq:rho-global-f}
\rho_{i\bar j}=\partial_i\partial_{\bar j}f.
\end{equation}

\step{Comparison at the extrema of $f$.}
At a minimum of $f$, $\rho$ is semipositive; at a maximum it is seminegative. By \eqref{eq:skew-curvature}, the inertia of $2\rho-sI$ is constant, since $2K^R$ is the transpose of $sI-2\rho$. If $\rho\ne0$ is semipositive, order its eigenvalues as $r_1\geq r_2\geq r_3\geq0$. The eigenvalues of $2\rho-sI$ are
\[
r_1-r_2-r_3,\qquad r_2-r_1-r_3,\qquad r_3-r_1-r_2.
\]
Its possible inertias, written as (positive, negative, zero), are
\[
(1,2,0),\quad(0,3,0),\quad(0,2,1),\quad(0,1,2).
\]
For a nonzero seminegative $\rho$, the positive and negative entries are interchanged. The two lists are disjoint. If $\rho=0$ at either extremum, then $K^R=0$ there, also contrary to its positive constant rank. Thus $r>0$ is impossible, and \eqref{eq:skew-curvature} gives $R=0$. This completes the proof of the proposition.
\end{proof}

\vspace{0.3cm}

\section{Geometry of the rank-two locus}\label{sec:geometry}
We next describe the geometry that remains if $R$ is not identically zero. The first observation supplies the continuation needed for the torsion tensor.

\begin{lemma}\label{lem:analytic}
The coefficients of $g$ are real analytic in every holomorphic coordinate chart. Consequently, a tensor formed analytically from $g$, its inverse, and finitely many derivatives will vanish identically if it vanishes on a nonempty open subset of $M$.
\end{lemma}
\begin{proof}
In a holomorphic chart,
\[
R_{i\bar jk\bar l}=-\partial_i\partial_{\bar j}g_{k\bar l}+Z_{i\bar jk\bar l},\qquad
Z_{i\bar jk\bar l}=\Gamma^a_{ik}\overline{\Gamma^b_{jl}}g_{a\bar b}.
\]
Contract \eqref{eq:polarization} by setting $j=i$ and summing with the fixed Euclidean matrix of the chart, rather than with $g^{-1}$. We obtain $L_0g=F_0(g,dg)$, where
\begin{equation}\label{eq:analytic-system}
(L_0H)_{k\bar l}=\sum_i\bigl(
\partial_i\partial_{\bar i}H_{k\bar l}
+\partial_k\partial_{\bar i}H_{i\bar l}
+\partial_i\partial_{\bar l}H_{k\bar i}
+\partial_k\partial_{\bar l}H_{i\bar i}\bigr).
\end{equation}
The right side is the corresponding sum of the four $Z$-terms and is analytic in the positive Hermitian matrix $g$ and its first derivatives. This is a determined system of nine real equations.

Let $P_\zeta$ be the positive principal symbol, obtained by replacing $\partial_i\partial_{\bar j}$ with $\zeta_i\bar\zeta_j$. By constant unitary equivariance and scaling, it suffices to take $\zeta=e_1$. For
\[
H=\begin{pmatrix}a&v^*\\v&D\end{pmatrix},\qquad a\in\R,\quad D=D^*,
\]
we have
\[
P_{e_1}H=\begin{pmatrix}4a+\tr D&2v^*\\2v&D\end{pmatrix},\qquad
\re\tr(H^*P_{e_1}H)=4a^2+a\tr D+4|v|^2+|D|^2.
\]
Since $|a\tr D|\leq a^2+|D|^2/2$, this is at least $|H|^2/2$. For a real Fourier covector $\xi$, $\zeta=(\xi_x-\sqrt{-1}\xi_y)/2$, so the lower bound is $|\xi|^2|H|^2/8$. The linearization of \eqref{eq:analytic-system} therefore has an elliptic principal part.

We apply the analytic regularity theorem for determined nonlinear elliptic systems, in its second-order form: a $C^{2,\alpha}$ solution of a real analytic system is analytic when its equations of variation are elliptic; see \cite[Theorem 6.7.6]{Morrey}. The hypotheses hold for the nine real entries of $g$ on a sufficiently small coordinate ball. Thus $g$ is real analytic. Holomorphic transition maps and the operations of inversion and differentiation preserve real analyticity. The last assertion follows from the analytic identity principle through overlapping connected charts. This completes the proof of the lemma.
\end{proof}

\begin{lemma}\label{lem:cubic-dense}
At a point where $S=0$ and $\nabla'S=0$, either $\B=0$ or $Q=0$. If $M$ is compact and $R\not\equiv0$, the sets $\{\B\ne0\}$ and $\{S\ne0\}$ are open and dense, and
\[
O_2:=\{\rank Q=2\}
\]
is open and dense.
\end{lemma}
\begin{proof}
At a fixed point put $q_a(y)=Q_{abc}y_by_c$, $b(y)={}^ty\B y$, and $s(y)=S(y,y,y)$. Equations~\eqref{eq:SG} and \eqref{eq:S-transport} give
\[
y\cdot q(y)=-\frac12s(y),\qquad
(\nabla_pS)(y,y,y)=\frac12\eta_ps(y)-2(q(y)\times\B y)_p.
\]
If $S=\nabla'S=0$, then $y\cdot q=0$ and $q\times\B y=0$. When $\B\ne0$, $b$ is a nonzero polynomial. On $\{b\ne0\}$, the second equation gives $q=\ell(y)\B y$, and the first gives $\ell(y)b(y)=0$. Thus $q=0$ there and, by polynomial uniqueness, $Q=0$.

If $\B$ vanishes on an open set, Lemmas~\ref{lem:analytic} and \ref{lem:pure-trace} give $R=0$. If $S$ vanishes on an open set, then so does $\nabla'S$. Either $\B$ vanishes there or the first assertion gives $Q=0$ on a smaller open set. Lemma~\ref{lem:continuation} and Proposition~\ref{prop:Qzero} again give $R=0$. These observations prove the two density assertions in the nonflat case. Propositions~\ref{prop:rank} and \ref{prop:Qzero}, together with Lemma~\ref{lem:continuation}, give the assertion about $O_2$.
\end{proof}

On $O_2$, define the plane bundle $P$ by $\operatorname{Im}Q=P\otimes\KM$. We emphasize that this is a bundle on the open dense subset $O_2$ instead of on the entire $M$.

\begin{lemma}\label{lem:contact}
The connection $D'=\nabla'-C$ preserves $P$. Let $\alpha$ be a nowhere-zero local section of $\Ann(P)$, and write $(D')^*\alpha=\lambda\otimes\alpha$. Put $a=\B(\alpha,\alpha)$ and $J_\alpha(X,Y)=\alpha(T_0(X,Y))$. Then
\begin{align}
\partial\lambda&=\frac32W,&\alpha(w)&=0,&
\nabla'_{\KM}a&=(2\lambda-\eta/2)\otimes a,& a\otimes W&=0,
\label{eq:contact-a}\\
\partial\alpha&=(\lambda-\eta)\wedge\alpha-J_\alpha,&
\alpha\wedge\partial\alpha&=-a.
\label{eq:contact-bracket}
\end{align}
Here three-forms are identified with sections of $\KM$. In particular, $W|_{\Lambda^2P}=0$, and $[P,P]\subset P$ at precisely the points where $a=0$. Moreover, for $X\in P$,
\begin{equation}\label{eq:plane-B}
\nabla_X\B=\frac12\eta(X)\B+w\odot X.
\end{equation}
\end{lemma}
\begin{proof}
The equation $\nabla'Q=CQ$ shows that $D'$ preserves its image; the scalar canonical connection does not change the plane $P$. The compatibility endomorphism $J$ in \eqref{eq:J-Q} annihilates $P$, so $J_{pq}=v_{pq}\otimes\alpha$ for uniquely determined vectors $v_{pq}$. Since $\tr J=3W/2$, the induced dual line connection gives $\partial\lambda=3W/2$.

We retain the unrestricted endomorphism equation before imposing a kernel condition. The expansion used in \eqref{eq:L-compat} gives, for every vector $b$,
\begin{equation}\label{eq:J-full}
J_{pq}b=-(\nabla_bT)(e_p,e_q)+\frac12\eta(b)T(e_p,e_q)
+W_{pq}b+\frac12\{U_q(b)e_p-U_p(b)e_q\},
\end{equation}
where $U_q(b)=(\nabla_q\eta)(b)+\eta(T(e_q,b))-\eta_q\eta(b)/2$. Let
\[
V_{la}=\frac12\sum_{p,q}\eps_{pql}v^a_{pq},\qquad C_0=\frac12(V+{}^tV),\qquad c=\alpha(w).
\]
The canonical factors are understood. The trace equation says $V\alpha=3w/2$, and hence $C_0(\alpha,\alpha)=3c/2$. Alternating-symbol contraction and symmetrization of \eqref{eq:J-full} give
\begin{equation}\label{eq:normal-B-remainder}
\nabla_b\B=\frac12\eta(b)\B+w\odot b-\alpha(b)C_0.
\end{equation}
The $U$-term disappears because it is skew in the two resulting indices. This proves \eqref{eq:plane-B} and gives
\[
(\nabla_b\B)(\alpha,\alpha)=\frac12\eta(b)a-\frac12c\alpha(b).
\]
The dual normal equation is
\begin{equation}\label{eq:alpha-dual}
\nabla_b\alpha=(\lambda(b)-\eta(b)/2)\alpha-\alpha\circ T_0(b,\cdot).
\end{equation}
The last covector annihilates $\B(\alpha,\cdot)$. Indeed, writing $\B=H\otimes\nu$ and $\beta=H(\alpha,\cdot)$ gives $\alpha(T_0(b,x))=\nu(b,x,\beta)$, which vanishes at $x=\beta$. Thus
\begin{equation}\label{eq:a-unrestricted}
\nabla'_{\KM}a=(2\lambda-\eta/2)\otimes a-\frac12c\alpha.
\end{equation}

For a covector, the exterior derivative includes its torsion term. Substituting \eqref{eq:alpha-dual} in
\[
\partial\alpha(X,Y)=(\nabla_X\alpha)(Y)-(\nabla_Y\alpha)(X)+\alpha(T(X,Y))
\]
gives the first formula of \eqref{eq:contact-bracket}. Alternating-symbol inversion gives $\alpha\wedge J_\alpha=a$ and $\alpha\wedge W=c$, proving the second formula.

On $\{a\ne0\}$ put $r=c/a$. Equation~\eqref{eq:a-unrestricted} becomes $\nabla'_{\KM}a=(2\lambda-\eta/2-r\alpha/2)\otimes a$. Its pure curvature is zero, so
\[
0=\frac52W-\frac12\partial(r\alpha).
\]
Wedge with $\alpha$ to get $0=5c/2+ra/2=3c$. On the interior of $\{a=0\}$, equation \eqref{eq:a-unrestricted} also gives $c=0$. These two open sets have dense union, so $c=0$ everywhere by continuity. Equation~\eqref{eq:a-unrestricted} now becomes homogeneous; taking its exterior derivative gives $a\otimes W=0$. This proves \eqref{eq:contact-a}, including at zeros of $a$.

Finally, $\alpha(w)=0$ is equivalent to $W|_{\Lambda^2P}=0$. For $X,Y\in P$, preservation by $D'$ gives
\[
[X,Y]\bmod P=(C_XY-C_YX-T(X,Y))\bmod P=T_0(X,Y)\bmod P.
\]
Its normal component vanishes exactly when $\alpha\wedge J_\alpha=a$ vanishes.
\end{proof}

We record a matrix estimate that will be used more than once.
\begin{lemma}\label{lem:block-bound}
At a point of $O_2$, choose a unitary frame with $P=\Span\{e_1,e_2\}$. Suppose $G$ has either form
\[
G=\begin{pmatrix}A&v\\0&-s\end{pmatrix}
\quad\hbox{or}\quad
G=\begin{pmatrix}A&0\\r&-s\end{pmatrix},\qquad
s=\tr A,\quad\det A=0.
\]
Then $r_G\leq3|S|^2$.
\end{lemma}
\begin{proof}
Put $H_0=A-sI/2$. Then $\tr G^2=2s^2$ and $\tr H_0^2=s^2/2$. Since $|\tr H_0^2|\leq|H_0|^2$, we get $r_G\leq4|H_0|^2$. Write $A=\left(\begin{smallmatrix}a&b\\d&e\end{smallmatrix}\right)$. The zero third row of $Q$ and \eqref{eq:Q-decomp} give
\[
S_{311}=\frac23d,\qquad S_{312}=\frac13(e-a),\qquad S_{322}=-\frac23b.
\]
Their multiplicities in the ordered norm are $3,6,3$. Therefore $|S|^2\geq\frac43(|b|^2+|d|^2)+\frac23|e-a|^2=\frac43|H_0|^2$, as required.
\end{proof}

\begin{proposition}\label{prop:zero-curl}
If $W=\partial\eta$ vanishes on a nonempty open subset, then $R=0$ on $M$.
\end{proposition}
\begin{proof}
Suppose, to the contrary, that $R\not\equiv0$. Lemma~\ref{lem:analytic} gives $W=0$ globally. On $O_2$, choose a local unit normal $n$ to $P$ and put $\alpha=(\cdot,n)$. By \eqref{eq:plane-B}, $\Nt_X=0$ for $X\in P$, so $\Nt=\alpha\otimes C_0$, with $C_0=\Nt_n$. The divergence identity in Lemma~\ref{lem:differential} gives $(C_0)_{a3}=0$ in an adapted unitary frame. Hence $C_0\in\Sym^2P\otimes\KM$.

Set $b(y)=\B(y,y)$, $c(y)=C_0(y,y)$,
\[
g_p(y)=\sum_j(G^*)_{pj}y_j,\qquad
\ell_p(y)=-2\sum_j(\nabla_{\bar j}\alpha)_p y_j\quad (p=1,2).
\]
The complete mixed equation \eqref{eq:N-mixed}, including the derivative of $\alpha$, yields
\begin{equation}\label{eq:Wzero-polynomial}
b(y)g_p(y)=c(y)\ell_p(y),\qquad p=1,2.
\end{equation}
If $\B\ne0$ and $C_0=0$, then $g_1=g_2=0$, and $\tr G=0$ gives $G^2=0$. If $\B,C_0$ are independent, any nonzero linear $g_p$ divides $c$. Indeed, restrict \eqref{eq:Wzero-polynomial} to $g_p=0$. Either $c$ or $\ell_p$ vanishes on that hyperplane; the latter possibility makes $\ell_p$ proportional to $g_p$ and then $b$ proportional to $c$, a contradiction. As $c$ depends only on $y_1,y_2$, $g_p$ has no $y_3$ term. Also any two nonzero $g_1,g_2$ are proportional: otherwise $g_1\ell_2=g_2\ell_1$ forces $\ell_1$ proportional to $g_1$, with the same contradiction. Thus $G$ has the first block form of Lemma~\ref{lem:block-bound}, and $r_G\leq3|S|^2$.

By Lemma~\ref{lem:cubic-dense}, $S$ and $\B$ are nonzero somewhere. Proposition~\ref{prop:compact-estimate} applies, and implies that
\[
V_{\mathrm{ex}}=\{t>0,\ r_G>6(|S|^2+|\Nt|^2)\}
\]
is nonempty and open, since $\int_Mt^3>0$. Intersect with the dense open set $O_2$ and shrink to a coordinate neighborhood. The preceding estimate shows that $C_0=\lambda_0\B$ at every point of this neighborhood, with $\lambda_0\ne0$. The coefficient $\lambda_0=(C_0,\B)/t$ is smooth. Thus
\begin{equation}\label{eq:Wzero-recurrence}
\Nt=\nu\otimes\B,\qquad \nu=\lambda_0\alpha\ne0,\qquad \operatorname{Im}\B\subset P
\end{equation}
throughout this open set.

Insert \eqref{eq:Wzero-recurrence} in \eqref{eq:N-mixed}. It follows that $S=I_\B(V)$ for a unique smooth vector $V$. The uniqueness follows from injectivity of $I_\B$: on diagonal variables its value is the product $b(y)V(y)$, and a polynomial ring has no zero divisors. A smooth left inverse is $(I_\B^*I_\B)^{-1}I_\B^*$. The remaining mixed equation is
\begin{equation}\label{eq:Wzero-nu}
(\nabla_{\bar j}\nu)_p+\nu_pV^j=-\frac12(G^*)_{pj}.
\end{equation}
Regard $\Gamma_{p\bar j}=(G^*)_{pj}=\Phi_{p\bar j}-2\rho_{p\bar j}$ as an $E^*$-valued $(0,1)$-form. In holomorphic coordinates, $\Phi_{p\bar j}=\partial_{\bar j}\eta_p$ and $\rho_{p\bar j}=-\partial_p\partial_{\bar j}\log\det g$, so $D''\Gamma=0$. Put $\beta(\bar Y)=(V,Y)$. Applying $D''$ to \eqref{eq:Wzero-nu} gives
\[
\Gamma\wedge\beta=2\nu\otimes\bp\beta.
\]
Choose a point of the neighborhood where $S\ne0$, which is possible by Lemma~\ref{lem:cubic-dense}. Then $\beta\ne0$. The first two rows of $\Gamma$ are proportional to $\beta$, because $\nu$ annihilates $P$. Hence the upper-left block of $G$ has determinant zero. Moreover, \eqref{eq:Wzero-recurrence} and $S=I_\B(V)$ give $S_{313}=S_{323}=0$. Since $Q_{313}=Q_{323}=0$, formula \eqref{eq:Q-decomp} gives $G_{13}=G_{23}=0$. Thus the second block form in Lemma~\ref{lem:block-bound} holds, and $r_G\leq3|S|^2$, contradicting the definition of $V_{\mathrm{ex}}$.
\end{proof}

\begin{proposition}\label{prop:flag}
Suppose $R\not\equiv0$. On $O_2$, let $n$ be a local unit normal, $\alpha=(\cdot,n)$, and
\[
\tau_n(X)=(T_0(X,n),n),\qquad A_n(X)=\pi_P\nabla_Xn.
\]
The form $\tau_n$ is independent of the phase of $n$. One has
\begin{equation}\label{eq:flag-basic}
\B(\alpha,\alpha)=0,\quad \alpha(w)=0,\quad [P,P]\subset P,
\quad \nabla_XY\in P\quad(X,Y\in P).
\end{equation}
For $X\in P$,
\begin{equation}\label{eq:flag-transport}
\nabla_Xw=(\eta(X)+\tau_n(X))w,\qquad
\nabla_X\tau_n=\frac12\eta(X)\tau_n-\tau_n(A_n(X))\alpha.
\end{equation}
The set $O_*:=\{x\in O_2:W\ne0,\ \tau_n\ne0\}$ is open and dense. On $O_*$ the lines $\Span(w)\otimes\KM^{-1}$ and $\ker(\tau_n|_P)$ are preserved by the Chern derivatives in directions of $P$.
\end{proposition}
\begin{proof}
By Proposition~\ref{prop:zero-curl}, $\{W\ne0\}$ is open and dense. Equations~\eqref{eq:contact-a} give $\B(\alpha,\alpha)=0$ on $O_2$ by continuity, and Lemma~\ref{lem:contact} gives bracket closure. Write $\tau=\tau_n$. Splitting into plane and normal inputs gives
\begin{align}
\alpha(T_0(X,Y))&=\tau(X)\alpha(Y)-\tau(Y)\alpha(X),\label{eq:normal-torsion}\\
T_0(X,Y)&=\tau(Y)X-\tau(X)Y\quad(X,Y\in P).\label{eq:plane-torsion}
\end{align}
For example, with $e_3=n$, $\B_{33}=0$, $\tau(e_1)=-\B_{23}$, and $\tau(e_2)=\B_{13}$; alternating-symbol inversion proves both formulas. Since $D'$ preserves $P$, \eqref{eq:plane-torsion} gives $\nabla_XY\in P$.

By \eqref{eq:alpha-dual} and \eqref{eq:normal-torsion},
\begin{equation}\label{eq:normal-equation}
\nabla_X\alpha=k(X)\alpha+\alpha(X)\tau,\qquad
k=\lambda-\eta/2-\tau.
\end{equation}
If $\tau$ vanished on an open set, this would be a scalar recurrence with pure curvature $\partial(\lambda-\eta/2)=W$. It would give $W\alpha=0$ there, contrary to Proposition~\ref{prop:zero-curl}. Thus $O_*$ is open and dense.

Apply the pure curvature of $\nabla'-\eta/2$ on $\Sym^2E\otimes\KM$ to \eqref{eq:plane-B}, using $X,Y\in P$ and $W(X,Y)=0$. This gives
\[
0=(\nabla_Xw-\eta(X)w/2)\odot Y
-(\nabla_Yw-\eta(Y)w/2)\odot X+w\odot T(X,Y).
\]
By \eqref{eq:plane-torsion}, the tensor $H(X)=\nabla_Xw-(\eta(X)+\tau(X))w$ satisfies $H(X)\odot Y=H(Y)\odot X$. Comparing the $11,22,13,23,12$ components for $X=e_1,Y=e_2$ shows that $H(X)=\kappa X$ on $P$, for a canonical-valued scalar $\kappa$.

The exterior equation $\partial W=0$ gives the full divergence
\begin{equation}\label{eq:w-div}
\sum_i(\nabla_iw)^i=\eta(w).
\end{equation}
Indeed, the torsion terms in its evaluation on $(e_1,e_2,e_3)$ have coefficients $-\eta_1,-\eta_2,-\eta_3$. Differentiating $\alpha(w)=0$ along $n$ and using \eqref{eq:normal-equation} gives $\alpha(\nabla_nw)=-\tau(w)$. The plane part of the divergence is $\eta(w)+\tau(w)+2\kappa$. Equation~\eqref{eq:w-div} therefore gives $\kappa=0$, proving the first formula in \eqref{eq:flag-transport}.

To differentiate $\tau$, fix a point and choose the phase of $n$ so that its unitary line connection vanishes at that point. For $X\in P$, metric compatibility and \eqref{eq:normal-equation} give $\nabla_{\bar X}n=0$ and $\nabla_Xn=A_n(X)$ there. Subtracting the derivative of the free $Y$-slot, we obtain
\begin{align*}
(\nabla_X\tau)(Y)
={}&((\nabla_XT_0)(Y,n),n)
+(T_0(Y,A_n(X)),n)\\
&+(T_0(Y,n),\nabla_{\bar X}n).
\end{align*}
The last term is zero. Formula~\eqref{eq:plane-B} makes the first term $\eta(X)\tau(Y)/2$, since its extra output lies in $P$. Formula~\eqref{eq:normal-torsion} makes the middle term $-\tau(A_n(X))\alpha(Y)$. This proves the second formula in \eqref{eq:flag-transport}. The assertions about the two lines now follow by differentiating their defining equations. In particular,
\begin{equation}\label{eq:alignment-transport}
\nabla_X^{\KM}\bigl(\tau(w)\bigr)
=\bigl(3\eta(X)/2+\tau(X)\bigr)\tau(w),\qquad X\in P.
\end{equation}
This completes the proof of the proposition.
\end{proof}

It is useful to express the zero condition for $\tau(w)$ without choosing a normal. Put
\begin{equation}\label{eq:global-detectors}
\begin{split}
&A_Q=QQ^*,\qquad \sigma=\tfrac12\{(\tr A_Q)^2-\tr(A_Q^2)\},\\
&J_Q=A_Q^2-(\tr A_Q)A_Q+\sigma I,\\
& \theta_Q(X)=\tr\bigl(J_Q\circ T_0(X,\cdot)\bigr),\qquad a_Q=\theta_Q(w).
\end{split}
\end{equation}
These are global analytic tensors. At a rank-two point the eigenvalues of $A_Q$ are $\lambda_1,\lambda_2,0$, with $\lambda_1,\lambda_2>0$. Thus
\[
J_Q=\sigma\pi_{P^\perp},\qquad \sigma=\lambda_1\lambda_2>0,
\qquad\theta_Q=\sigma\tau_n,
\]
whereas $J_Q=\theta_Q=a_Q=0$ at rank at most one. Consequently either $a_Q$ is identically zero, or its nonzero set is open and dense. Only the polynomial tensors in \eqref{eq:global-detectors}, not a normalized normal or projector, are used at a rank drop.

\vspace{0.3cm}

\section{The case \texorpdfstring{$\det\B\not\equiv0$}{det B not identically zero}}\label{sec:full}
Throughout this section, we will assume that $R\not\equiv0$ and we want to derive a contradiction. Set $\delta=\det\B$. Since $\B\in\Sym^2\!E\otimes\KM$, its determinant is a section of $(\det E)^2\otimes\KM^3=\KM$. First let us record the mixed equation on the common regular set. Choose a unit normal $n$ to $P$, put $\alpha=(\cdot,n)$, and write $\tau=\tau_n$. Define
\[
C_0=\Nt_n-w\odot n.
\]
Equation~\eqref{eq:plane-B} and the divergence formula in Lemma~\ref{lem:differential} give
\begin{equation}\label{eq:universal-remainder}
\Nt_X=w\odot X+\alpha(X)C_0,\qquad C_0(\alpha,\cdot)=-\frac32w.
\end{equation}
In an adapted unitary frame $e_3=n$, this says $(C_0)_{a3}=-3w_a/2$, while $\B_{33}=w_3=(C_0)_{33}=0$. For formal covector variables $y$, put
\[
b(y)=\B(y,y),\quad c(y)=C_0(y,y),\quad j(y)=\Sym_2(\nabla''w)(y,y),
\]
\[
\ell_p(y)=\sum_j(\nabla_{\bar j}\alpha)_p y_j,\qquad
g_p(y)=\sum_j(G^*)_{pj}y_j.
\]
The mixed equation \eqref{eq:N-mixed} becomes
\begin{equation}\label{eq:universal-mixed}
y_pj(y)+c(y)\ell_p(y)=-\frac12b(y)g_p(y),\qquad p=1,2.
\end{equation}
When $\delta\ne0$, $b$ is an irreducible quadratic, since the symmetric matrix of a product of two linear forms has rank at most two. If $\B,C_0$ are independent, eliminating $j$ and using unique factorization gives
\begin{equation}\label{eq:mixed-syzygy}
y_2\ell_1-y_1\ell_2=\lambda_0b,\qquad
y_2g_1-y_1g_2=-2\lambda_0c
\end{equation}
for a scalar $\lambda_0$, including the possibility $\lambda_0=0$.

\begin{lemma}\label{lem:scalar-pattern}
There is no nonempty open subset of $O_2$ on which an adapted unitary frame satisfies
\begin{equation}\label{eq:scalar-pattern}
G=\begin{pmatrix}m_0&0&v_1\\0&m_0&v_2\\0&0&-2m_0\end{pmatrix},\quad
\pi_P\nabla_{e_p}n=\bar r_0e_p\ (p=1,2),\quad
\pi_P\nabla_nn=0,
\end{equation}
with $(\B_{13},\B_{23})\ne(0,0)$.
\end{lemma}
\begin{proof}
Write $u_1=\B_{13}$ and $u_2=\B_{23}$. The complete derivative of the two-slot tensor $G$ is
\[
(\nabla_pG)_{rj}=e_p(G_{rj})-\Gamma^k_{pr}G_{kj}+\Gamma^j_{pk}G_{rk}.
\]
At $(p,r,j)=(1,3,1),(2,3,2),(1,2,1),(2,1,2)$, its values under \eqref{eq:scalar-pattern} are
\[
-3m_0\bar r_0,\quad-3m_0\bar r_0,\quad\bar r_0v_2,\quad\bar r_0v_1.
\]
On the other hand, differentiating $G_{rj}=2\eps_{alr}Q_{alj}$ by \eqref{eq:Q-transport} gives
\begin{equation}\label{eq:G-transport}
(\nabla_pG)_{rj}=\frac12\eta_pG_{rj}
+2\sum_{a,l,b,t}\eps_{alr}\eps_{pbt}\B_{at}Q_{blj}.
\end{equation}
Since $Q_3=0$, we have $Q_{231}=m_0/2$, $Q_{132}=-m_0/2$, $Q_{121}=Q_{211}$, and $Q_{122}=Q_{212}=Q_{221}$. Set $D_0=u_1Q_{221}-u_2Q_{211}$. The four values furnished by \eqref{eq:G-transport} are $2D_0,-2D_0,-u_1m_0,u_2m_0$. Comparing gives
\[
m_0\bar r_0=0,\qquad \bar r_0v_2=-u_1m_0,\qquad \bar r_0v_1=u_2m_0.
\]
If $m_0\ne0$, these equations force $r_0=u_1=u_2=0$, a contradiction. Thus $m_0=0$ and $G^2=0$ on the open set. Analytic continuation gives $G^2=0$ on $M$. Lemma~\ref{lem:cubic-dense} and \eqref{eq:compact-estimate} now contradict $\int_Mt^3>0$. This completes the proof of the lemma.
\end{proof}

\begin{lemma}\label{lem:alignment}
On $O_2$ one has $\tau(w)=0$. Moreover,
\begin{equation}\label{eq:det-plane}
\nabla_X^{\KM}\delta=\frac32\eta(X)\delta,\qquad X\in P.
\end{equation}
\end{lemma}
\begin{proof}
Suppose first that $\delta\ne0$ and $a_{\mathrm{rel}}:=\tau(w)\ne0$ on an open set. With $u_i=\B_{i3}$, we have $\tau=(-u_2,u_1,0)$ and $a_{\mathrm{rel}}=u_1w_2-u_2w_1$. The third-column identity in \eqref{eq:universal-remainder} prevents $C_0$ from being proportional to $\B$. Thus \eqref{eq:mixed-syzygy} applies. Its $y_2y_3$ and $y_1y_3$ coefficients give
\[
\ell_{13}=2\lambda_0u_2,\quad\ell_{23}=-2\lambda_0u_1,
\quad g_{13}=6\lambda_0w_2,\quad g_{23}=-6\lambda_0w_1.
\]
Differentiate $\alpha(w)=0$ in the barred normal direction. If $H_w=\Sym_2\nabla''w$, this gives $(H_w)_{33}=-\ell_{13}w_1-\ell_{23}w_2=2\lambda_0a_{\mathrm{rel}}$. The $y_1y_3^2$ coefficient of \eqref{eq:universal-mixed} for $p=1$ instead gives
\[
(H_w)_{33}+2(C_0)_{13}\ell_{13}=-\B_{13}g_{13},
\]
so $(H_w)_{33}=-6\lambda_0a_{\mathrm{rel}}$. Hence $\lambda_0=0$. The remaining coefficients of \eqref{eq:mixed-syzygy} give $\ell_{pj}=r_0\delta_{pj}$ and $g_{pj}=\mu_0\delta_{pj}$ for $p=1,2$. Taking adjoints and using $\tr G=0$ and metric compatibility gives \eqref{eq:scalar-pattern}, which is impossible by Lemma~\ref{lem:scalar-pattern}.

Before imposing alignment, differentiating the determinant by its polynomial adjugate gives
\begin{equation}\label{eq:det-before-alignment}
\nabla_X^{\KM}\delta=\frac32\eta(X)\delta-\tau(X)\tau(w),\qquad X\in P.
\end{equation}
Indeed, in a frame with $\B_{33}=0$, the leading plane block of $\adj\B$ is $-\tau\otimes\tau$. Its contraction with the $w\odot X$ term in \eqref{eq:plane-B} is $-\tau(X)\tau(w)$. This proof uses no inverse of $\B$.

If $\delta\equiv0$, equation \eqref{eq:det-before-alignment} and $\tau\ne0$ on $O_*$ give $\tau(w)=0$ there, and then on $O_2$ by continuity. If $\delta\not\equiv0$, its nonzero set is open and dense by Lemma~\ref{lem:analytic}. Any open nonalignment set would meet it, contrary to the first part of the proof. Thus alignment holds in both cases, and \eqref{eq:det-plane} follows.
\end{proof}

Now assume $\delta\not\equiv0$ and put
\[
U=O_*\cap\{\delta\ne0\}.
\]
This is a nonempty open dense set. All inverses and normalized fields in the rest of this section are restricted to $U$.

\begin{proposition}\label{prop:normalization}
There is an intrinsic smooth frame $F_1,F_2,F_3$ on $U$, not necessarily unitary or holomorphic, with dual coframe $\varphi_1,\varphi_2,\varphi_3$ and $\nu=\varphi_1\wedge\varphi_2\wedge\varphi_3$, such that
\begin{equation}\label{eq:full-B-normal}
\delta=-4\nu,\quad w=-4F_1\otimes\nu,\quad
\B=\begin{pmatrix}u&2z-3&-2\\2z-3&1&0\\-2&0&0\end{pmatrix}\otimes\nu.
\end{equation}
Writing $\eta_i=\eta(F_i)$, the column-coefficient connection matrices $\Gamma_i=\Omega(F_i)$ are
\begin{equation}\label{eq:full-connection}
\begin{aligned}
\Gamma_1&=-\eta_1I/2,\qquad
\Gamma_2=-\eta_2I/2+\diag(-2,0,2),\\
\Gamma_3&=-\eta_3I/2+\begin{pmatrix}7-4z&0&2a\\1&5-2z&2b\\0&2&4z-5\end{pmatrix}.
\end{aligned}
\end{equation}
Here $z,u,a,b$ are smooth complex functions. Furthermore,
\begin{equation}\label{eq:full-brackets}
\begin{aligned}
{}[F_1,F_2]&=4F_1,&[F_1,F_3]&=(6z-10)F_1,\\
[F_2,F_3]&=-uF_1-2F_2+2F_3,\\
F_1z&=0,&F_2z&=2z-4,&F_1u&=-4,&F_2u&=4u.
\end{aligned}
\end{equation}
Let $H_{pq}=(F_q,F_p)$, let $H_0$ be its leading $2\times2$ block, and set
\begin{equation}\label{eq:Gram-def}
\begin{gathered}
\binom rs=H_0^{-1}\binom{H_{13}}{H_{23}},\qquad N_0=F_3-rF_1-sF_2,\qquad d=(N_0,N_0)>0,\\
D_0=2z-3-2s,\qquad U_0=u-4r,\qquad
J_0=\begin{pmatrix}D_0&-U_0\\1&-D_0\end{pmatrix},\\
A_0=\begin{pmatrix}-F_1r&4r-F_2r\\-F_1s&2s-F_2s\end{pmatrix}.
\end{gathered}
\end{equation}
If $v$ is the column of $\pi_P\nabla_{N_0}N_0$ in $F_1,F_2$, and $m=(H_0v)_2/(4d)$, then
\begin{equation}\label{eq:Gram-motion}
(A_0)_{\rm tf}=mJ_0^{*_{H_0}},\qquad H_0v=4dm\binom01,
\end{equation}
where $X^\dagger=\overline X^{\,t}$, $X^{*_{H_0}}=H_0^{-1}X^\dagger H_0$, and ${\rm tf}$ denotes the trace-free part. The set $\{m\ne0\}$ is open and dense in $U$.
\end{proposition}
\begin{proof}
\step{The inverse form and the normal remainder.}
Let $\beta=\B^{-1}w$. By alignment, $w=q\B(\alpha,\cdot)$ with $q\ne0$, and $\beta=q\alpha$. Equation~\eqref{eq:normal-equation}, with its scalar form denoted here by $k_0$, gives
\[
\partial\alpha=(k_0-\eta/2)\wedge\alpha,
\qquad Xq=(\eta(X)/2+\tau(X)-k_0(X))q\quad(X\in P).
\]
The torsion term is included in the first equality. It follows that
\begin{equation}\label{eq:beta-full}
\partial\beta=q\tau\wedge\alpha=W,\qquad \beta(w)=0.
\end{equation}
For example, $w=q(\B_{13},\B_{23},0)$ gives $W_{13}=-q\B_{23}$ and $W_{23}=q\B_{13}$.

Set $D=C_0/q+3\B/2$. Under a unit phase change of $n$, both $C_0$ and $q$ acquire the same phase, so $D$ is intrinsic. Equation~\eqref{eq:universal-remainder} gives $D(\beta,\cdot)=0$; thus $D$ is a symmetric plane tensor. The full derivative becomes
\begin{equation}\label{eq:full-B-recurrence}
\nabla_X\B=(\eta(X)/2-3\beta(X)/2)\B+w\odot X+\beta(X)D.
\end{equation}
The normal equation for $\beta$ has the form
\begin{equation}\label{eq:beta-recurrence}
\nabla_X\beta=(\eta(X)/2+\tau(X)+\kappa\beta(X))\beta+\beta(X)\tau
\end{equation}
for a unique smooth scalar $\kappa$, since its remaining scalar one-form vanishes on $P=\ker\beta$.

The independence locus of $\B,C_0$ is open and dense in $U$. Otherwise $C_0=c_0\B$ on an open set, and its third column gives $c_0=-3q/2$, so $D=0$. The determinant equation would then be
\[
\nabla'_{\KM}\delta=(3\eta/2-7\beta/2)\delta.
\]
Its pure curvature gives $0=\partial(3\eta/2-7\beta/2)=-2W$, a contradiction.

\step{Two canonical-valued plane vectors.}
Put
\[
v_0=\B(\tau,\cdot),\qquad V=(\kappa+1/2)w+2v_0.
\]
In a unitary frame with $e_1$ along $w$ and $e_3=n$, alignment gives
\begin{equation}\label{eq:full-unitary-B}
\B=\begin{pmatrix}A&d_0&b_0\\d_0&e_0&0\\b_0&0&0\end{pmatrix}\otimes\nu_0,
\quad w=s_0e_1\otimes\nu_0,\quad
\tau=b_0\varphi_2,\quad\beta=q\varphi_3,\quad q=s_0/b_0,
\end{equation}
where $s_0b_0e_0\ne0$. Thus $v_0=b_0(d_0e_1+e_0e_2)\otimes\nu_0$ and $V$ is independent of $w$. Differentiating $w=\B(\beta,\cdot)$ gives
\begin{equation}\label{eq:full-w-recurrence}
\nabla_Xw=[\eta(X)+\tau(X)+(\kappa-1)\beta(X)]w+\beta(X)v_0.
\end{equation}

Let $a_0=(\eta-3\beta)/2$. The pure compatibility of \eqref{eq:full-B-recurrence} is
\begin{align*}
0={}&\partial a_0(X,Y)\B-a_0(X)[w\odot Y+\beta(Y)D]
+a_0(Y)[w\odot X+\beta(X)D]\\
&+(\nabla_Xw)\odot Y-(\nabla_Yw)\odot X+w\odot T(X,Y)\\
&+\partial\beta(X,Y)D-\beta(X)\nabla_YD+\beta(Y)\nabla_XD.
\end{align*}
Here $\partial a_0=-W$ and $\partial\beta=W$. Take $X\in P$ and $Y=n$, and substitute \eqref{eq:full-w-recurrence}. Since $W(X,n)=q\tau(X)$, this gives
\begin{align*}
\nabla_XD={}&(\eta(X)/2-\tau(X))D+\tau(X)\B\\
&-q^{-1}w\odot[\tau(X)n+T_0(X,n)]
+[(\kappa+1/2)w+v_0]\odot X.
\end{align*}
The middle two terms equal $v_0\odot X$. This can be checked on $e_1,e_2$ in \eqref{eq:full-unitary-B}, using
\[
T_0(e_1,n)=-d_0e_1-e_0e_2,\qquad
T_0(e_2,n)=Ae_1+d_0e_2+b_0n.
\]
Therefore
\begin{equation}\label{eq:D-plane}
\nabla_XD=(\eta(X)/2-\tau(X))D+V\odot X,\qquad X\in P.
\end{equation}

The second identity of \eqref{eq:flag-transport} yields
\[
\nabla_Xv_0=\eta(X)v_0+\bigl[\tau(X)/2-\tau(A_n(X))/q\bigr]w,
\]
and hence
\[
\nabla_XV-\eta(X)V=\bigl[X\kappa+(\kappa+3/2)\tau(X)-2\tau(A_n(X))/q\bigr]w.
\]
This retains the normal-motion term. For $X,Y\in P$, the same torsion and normal formulas give $\partial\tau(X,Y)=W(X,Y)=0$. Apply pure compatibility to \eqref{eq:D-plane}. If $H(X)=\nabla_XV-\eta(X)V$, we obtain $H(X)\odot Y=H(Y)\odot X$. Each $H(X)$ is a multiple of $w$. Taking $X$ along $w$ and $Y$ independent of it gives $H=0$. Thus
\begin{equation}\label{eq:V-plane}
\nabla_XV=\eta(X)V,\qquad X\in P.
\end{equation}
Set $z=\tr(\B^{-1}D)$. Differentiating, using \eqref{eq:full-B-recurrence} and \eqref{eq:D-plane}, gives
\[
Xz=-\tau(X)z+\tr[\B^{-1}(V\odot X)]
-\tr[\B^{-1}(w\odot X)\B^{-1}D]=(2-z)\tau(X).
\]
The last trace vanishes because $D(\beta,\cdot)=0$, while $\B^{-1}V=(\kappa+1/2)\beta+2\tau$.

\step{The third vector and the normalization.}
By \eqref{eq:V-plane}, there is a unique smooth $Z\in E\otimes\KM$ with $\nabla_XV=\eta(X)V+\beta(X)Z$ for every $X$. Locally, $Z=q^{-1}(\nabla_nV-\eta(n)V)$. Put $Y=Z+V$. Exterior differentiation gives
\[
0=W\otimes V-\eta\wedge\beta\otimes Z+W\otimes Z-\beta\wedge\nabla'Z.
\]
Evaluating on $X\in P,n$ shows that
\[
\nabla_XY=(\eta(X)-\tau(X))Y\quad(X\in P).
\]
In \eqref{eq:full-unitary-B}, $\delta=-b_0^2e_0\nu_0$ and $\tau(V)=-2\delta$. Differentiating $\beta(V)=0$ gives $\beta(Z)=-\tau(V)=2\delta$. Thus $\beta(Y)=2\delta$, and comparison of the $e_2,e_3$ coefficients gives
\begin{equation}\label{eq:full-wedge}
w\wedge V\wedge Y=-4\delta^2.
\end{equation}
In particular these three sections are independent.

Put $k=\kappa/2-5/4$. There are smooth functions $a,b,c$ such that the full derivatives are
\begin{align}
\nabla_Xw&=(\eta(X)+\tau(X)+k\beta(X))w+\beta(X)V/2,\notag\\
\nabla_XV&=(\eta(X)-\beta(X))V+\beta(X)Y,\label{eq:full-triple}\\
\nabla_XY&=\beta(X)aw+\beta(X)bV+(\eta(X)-\tau(X)+c\beta(X))Y.\notag
\end{align}
Determinant differentiation in \eqref{eq:full-B-recurrence} gives
\[
\nabla'_{\KM}\delta=\ell\delta,\qquad \ell=3\eta/2+(z-7/2)\beta.
\]
Comparing the derivative of \eqref{eq:full-wedge} with \eqref{eq:full-triple} gives $c=2z-6-k$. Differentiating $\beta(Y)=2\delta$ then gives $\tau(Y)/(2\delta)=-z-k$.

Define $F_1=w/\delta$, $F_2=V/\delta$, $F_3=Y/\delta$. Their matching volume satisfies $\delta=-4\nu$, and the preceding contractions give
\[
\beta=2\varphi_3,\qquad \tau=-2\varphi_2-2(z+k)\varphi_3.
\]
The equation $\B(\beta,\cdot)=w$ fixes the third column of $\B$ as $(-2,0,0)^t$. Its determinant fixes the $22$ entry as $1$, and the $F_1$ coefficient of $V=(\kappa+1/2)w+2\B(\tau,\cdot)$ fixes the $12$ entry as $2z-3$. This proves \eqref{eq:full-B-normal}.

Divide \eqref{eq:full-triple} by $\delta$, subtracting $\ell$ in each column. The resulting column connection is
\[
\Omega=\Sigma I+\tau\diag(1,0,-1)
+\beta\begin{pmatrix}k&0&a\\1/2&-1&b\\0&1&2z-6-k\end{pmatrix},
\qquad\Sigma=-\eta/2-(z-7/2)\beta.
\]
Substituting the formulas for $\beta,\tau$ proves \eqref{eq:full-connection}. Alternating-symbol inversion with the matching volume gives
\[
T_0(F_1,F_2)=-2F_1,\quad
T_0(F_1,F_3)=-(2z-3)F_1-F_2,\quad
T_0(F_2,F_3)=uF_1+(2z-3)F_2-2F_3.
\]
In $[F_i,F_j]=\nabla_{F_i}F_j-\nabla_{F_j}F_i-T(F_i,F_j)$ the scalar $\eta$-terms cancel the trace torsion. This gives the brackets in \eqref{eq:full-brackets}. The formula for $Xz$ gives its first two derivatives. Finally, the $11$ entry of \eqref{eq:plane-B}, retaining the canonical connection, is $Xu=-2\tau(X)u-4\varphi_1(X)$ for $X\in P$. This gives the derivatives of $u$.

\step{The normal acceleration.}
On the dense independence locus of $\B,C_0$, write in \eqref{eq:full-unitary-B}
\[
C_0=\begin{pmatrix}F&H&-3s_0/2\\H&J&0\\-3s_0/2&0&0\end{pmatrix}\otimes\nu_0.
\]
The five coefficients of \eqref{eq:mixed-syzygy} give
\begin{equation}\label{eq:full-mixed-arrays}
\begin{aligned}
\bigl((\nabla_{\bar e_c}\alpha)(e_p)\bigr)_{p=1,2;\,c=1,2,3}
&=\begin{pmatrix}r_0+\lambda_0d_0&\lambda_0e_0&0\\-\lambda_0A&r_0-\lambda_0d_0&-2\lambda_0b_0\end{pmatrix},\\
G^*&=\begin{pmatrix}
\mu_0-2\lambda_0H&-2\lambda_0J&0\\
2\lambda_0F&\mu_0+2\lambda_0H&-6\lambda_0s_0\\
v_1&v_2&-2\mu_0
\end{pmatrix}.
\end{aligned}
\end{equation}
In particular $\pi_P\nabla_nn=-2\overline{\lambda_0b_0}e_2$. If this acceleration vanished on an open subset of $U$, that set would meet the independence locus and give $\lambda_0=0$ there. The arrays would then give \eqref{eq:scalar-pattern}, which Lemma~\ref{lem:scalar-pattern} excludes. Thus the acceleration is nonzero on an open dense set.

\step{The Gram-matrix formula.}
In the frame $(F_1,F_2,N_0)$, the metric is $\diag(H_0,d)$ and the canonical volume is still $\nu$. Since $\tau(N_0)=0$, we have $s=z+k$, and
\[
\B=\begin{pmatrix}U_0&D_0&-2\\D_0&1&0\\-2&0&0\end{pmatrix}\otimes\nu,
\qquad\alpha=\sqrt d\,\varphi_3,\qquad\beta=2\varphi_3.
\]
Here the coframe refers to $(F_1,F_2,N_0)$. Formula~\eqref{eq:full-connection} gives the two columns of $\pi_P\nabla_{F_i}N_0$ as those of $A_0$. For $n=N_0/\sqrt d$, the plane columns of $X\mapsto\pi_P\nabla_Xn$ are $T_P=A_0/\sqrt d$ and $T_N=v/\sqrt d$; differentiating $d^{-1/2}$ has only normal output.

Apply \eqref{eq:universal-mixed} in this nonunitary frame, raising the barred derivative slot with its actual metric. On the independence locus, the same divisibility argument gives
\[
y_2L_1(y)-y_1L_2(y)
=\lambda_1(U_0y_1^2+2D_0y_1y_2-4y_1y_3+y_2^2).
\]
Thus the two rows of raised coefficients are
\[
\begin{pmatrix}R_0+\lambda_1D_0&\lambda_1&0\\-\lambda_1U_0&R_0-\lambda_1D_0&4\lambda_1\end{pmatrix}.
\]
Before raising, metric compatibility gives the array ${}^tH_0\overline{(T_P,T_N)}$. Raising multiplies on the right by $\diag(H_0,d)^{-t}$. The first block is therefore ${}^t(T_P^{*_{H_0}})$ and the last column is $d^{-1}{}^tH_0\bar T_N$. Transposing the first block and using the last column yields
\[
(A_0)_{\rm tf}=\sqrt d\,\bar\lambda_1J_0^{*_{H_0}},\qquad
H_0v=4d^{3/2}\bar\lambda_1\binom01.
\]
Hence $m=\sqrt d\,\bar\lambda_1$, proving \eqref{eq:Gram-motion} on the independence locus. All its quantities are smooth on $U$, so the formula extends to $U$ by continuity. The acceleration is represented by $v/d$, and \eqref{eq:Gram-motion} shows that $v\ne0$ exactly when $m\ne0$. Its density was proved above. We have therefore completed the proof of the proposition.
\end{proof}

\vspace{0.1cm}

\begin{proposition}\label{prop:full-flat}
If $\det\B\not\equiv0$, then $R=0$.
\end{proposition}
\begin{proof}
Suppose $R\not\equiv0$ and use Proposition~\ref{prop:normalization}. Abbreviate $N=N_0$, $D=D_0$, $A=A_0$, and write
\[
\mu=\bar m,\qquad c=\tfrac12\tr A,\qquad
\mathsf T=A^{*_{H_0}}=\bar cI+\mu J_0.
\]
By \eqref{eq:full-brackets} and \eqref{eq:Gram-def},
\begin{equation}\label{eq:DU-derivatives}
F_1D=2A_{21},\quad F_2D=2D-2+2A_{22},\quad
F_1U_0=4A_{11}-4,\quad F_2U_0=4U_0+4A_{12}.
\end{equation}
Work in the frame $(F_1,F_2,N)$, whose Gram matrix is $\diag(H_0,d)$. Write its column connection as
\[
\Omega=\begin{pmatrix}\Pi&p\\q&k\end{pmatrix},
\qquad \Pi_i=\Pi(F_i),\quad \Pi_{\bar i}=\Pi(\bar F_i),
\quad\Pi_N=\Pi(N),\quad\Pi_{\bar N}=\Pi(\bar N).
\]
Let $\mathbf e_1,\mathbf e_2$ denote the standard coordinate columns of $\C^2$. Changing the third frame vector in \eqref{eq:full-connection}, and then using metric compatibility, gives
\begin{equation}\label{eq:full-blocks}
\begin{gathered}
\Pi_i=-\eta(F_i)I/2+L_i,\qquad L_1=0,\quad L_2=\diag(-2,0),\\
k_i:=k(F_i)=-\eta(F_i)/2+2\epsilon_i,
\qquad\epsilon_1=0,\quad\epsilon_2=1,\\
p(F_i)=A\mathbf e_i,\quad q(F_i)=0,\quad q(N)=2\mathbf e_2^t,\\
p(\bar F_i)=0,\quad q(\bar F_i)=-d^{-1}\mathbf e_i^tH_0\mathsf T,
\quad q(\bar N)=-4\mu\mathbf e_2^t.
\end{gathered}
\end{equation}
The last equality follows from $q(\bar N)=-d^{-1}v^\dagger H_0$ and \eqref{eq:Gram-motion}. Set $a_i=F_i(\log d)-k_i$. The remaining identities needed below are
\begin{equation}\label{eq:full-metric-identities}
F_iH_0=H_0\Pi_i+\Pi_{\bar i}^{\dagger}H_0,\qquad
 a_i=\overline{k(\bar F_i)},\qquad (\Pi_N)_{21}=1.
\end{equation}

\step{The mixed curvature.}
All curvature components in this proof refer to $(F_1,F_2,N)$ and its actual metric. The zero third row of $Q$, together with \eqref{eq:polarization}, gives
\begin{equation}\label{eq:plane-curvature-skew}
R_{i\bar jk\bar l}=-R_{i\bar lk\bar j},\qquad i,k\leq2.
\end{equation}
In particular $R_{i\bar3k\bar3}=0$, so the lower-left block of $R(F_i,\bar N)$ vanishes. The mixed torsion is zero, and \eqref{eq:full-blocks} gives the complete bracket
\begin{equation}\label{eq:Fi-N-bracket}
[F_i,\bar N]=a_i\bar N-\sum_{p=1}^2(\Pi_{\bar N})_{pi}F_p+4\mu\epsilon_iN.
\end{equation}
The lower-left curvature block is $dq+q\wedge\Pi+k\wedge q$. Its value on $(F_i,\bar N)$ is
\[
-4\{F_i\mu+(4\epsilon_i-a_i)\mu\}\mathbf e_2^t.
\]
Indeed, the bracket term contributes $4a_i\mu-8\epsilon_i\mu$, and the two connection-product terms contribute another $-8\epsilon_i\mu$. Its vanishing proves
\begin{equation}\label{eq:mu-transport}
F_i\mu=(a_i-4\epsilon_i)\mu.
\end{equation}

Put $K_i=\pi_PR(F_i,\bar N)|_P$. Stack the rows of the lower-left curvature on $(F_i,\bar F_j)$, $j=1,2$. Equation~\eqref{eq:plane-curvature-skew} makes this matrix $-d^{-1}H_0K_i$. The anti-plane, normal, and anti-normal coefficients of $[F_i,\bar F_j]$ are, respectively,
\[
\overline{(\Pi_{\bar i})_{pj}},\qquad
 d^{-1}(H_0\mathsf T)_{ji},\qquad -d^{-1}(H_0A)_{ji}.
\]
Thus the stacked matrix $q([F_i,\bar F_j])$ is
\[
-d^{-1}\Pi_{\bar i}^{\dagger}H_0\mathsf T
+2d^{-1}H_0\mathsf T\mathbf e_i\mathbf e_2^t
+4\mu d^{-1}H_0A\mathbf e_i\mathbf e_2^t.
\]
Differentiate $q(\bar F_j)=-d^{-1}\mathbf e_j^tH_0\mathsf T$ and use \eqref{eq:full-metric-identities} to cancel the derivatives of $H_0$. Multiplication by $-dH_0^{-1}$ gives
\begin{equation}\label{eq:Ki-full}
K_i=F_i\mathsf T+[\Pi_i,\mathsf T]-a_i\mathsf T
+2\mathsf T\mathbf e_i\mathbf e_2^t+4\mu A\mathbf e_i\mathbf e_2^t.
\end{equation}
Let $\ell=\bar c+(D+2)\mu$ and $b_i=F_i\bar c-a_i\bar c$. Using \eqref{eq:DU-derivatives} and \eqref{eq:mu-transport}, we obtain
\begin{align}
K_1&=b_1I+2\bar c\mathbf e_1\mathbf e_2^t
+\mu\begin{pmatrix}2A_{21}&2D+4\\0&2A_{21}+2\end{pmatrix},\notag\\
K_2&=b_2I+2\bar c\mathbf e_2\mathbf e_2^t
+\mu\begin{pmatrix}-2D-2+2A_{22}&0\\-2&2+2A_{22}\end{pmatrix}.
\label{eq:Ki-before-trace}
\end{align}
For instance, $[L_2,J_0]=\left(\begin{smallmatrix}0&2U_0\\2&0\end{smallmatrix}\right)$, and the upper-right terms cancel against the last two terms of \eqref{eq:Ki-full}. In particular,
\begin{equation}\label{eq:K21-first}
(K_2)_{21}=F_2\mu+2\mu-a_2\mu=-2\mu.
\end{equation}
The trace-free parts are
\[
(K_1)_{\rm tf}=\begin{pmatrix}-\mu&2\ell\\0&\mu\end{pmatrix},\qquad
(K_2)_{\rm tf}=\begin{pmatrix}-\ell&0\\-2\mu&\ell\end{pmatrix}.
\]

\step{The scalar traces and a curvature root.}
Evaluate the second Bianchi identity on $(F_1,F_2,\bar N)$. The pure curvature, $q(F_i)$, and the lower-left blocks of $R(F_i,\bar N)$ vanish. Using $[F_1,F_2]=4F_1$, its plane block is
\begin{equation}\label{eq:Ki-Bianchi}
F_1K_2-F_2K_1-[L_2,K_1]-4K_1-a_1K_2+a_2K_1=0.
\end{equation}
The $11$ and $12$ entries of its trace-free part, together with \eqref{eq:mu-transport}, give
\[
F_1\ell=a_1\ell,\qquad F_2\ell=(a_2-2)\ell.
\]
Differentiate $\ell=\bar c+(D+2)\mu$ and use \eqref{eq:DU-derivatives}. This gives $b_1=-2A_{21}\mu$ and $b_2=-2\bar c+(6-2A_{22})\mu$. Consequently the complete matrices in \eqref{eq:Ki-before-trace} are
\begin{equation}\label{eq:Ki-final}
K_1=\begin{pmatrix}0&2\ell\\0&2\mu\end{pmatrix},\qquad
K_2=\begin{pmatrix}-2\ell+8\mu&0\\-2\mu&8\mu\end{pmatrix}.
\end{equation}

By \eqref{eq:plane-curvature-skew} and reality, $R_{1\bar j1\bar l}=0$ for $j,l\leq2$. The first matrix in \eqref{eq:Ki-final}, together with the lower-left vanishing, gives $R(F_1,\bar N)F_1=0$. Equation~\eqref{eq:plane-curvature-skew} also makes the remaining normal component of $R(F_1,\bar F_j)F_1$ zero. Thus $R(F_1,\bar Z)F_1=0$ for every $Z$.

Write $Q^c=\bar Q:\Sym^2E\longrightarrow\overline{E\otimes\KM}$. In the present nonunitary frame, with its matching volume, alternating-symbol inversion reads
\begin{equation}\label{eq:Qc-inversion}
\frac12(R_{i\bar jk\bar l}+R_{k\bar ji\bar l})
=\sum_a\eps_{jla}(Q^c)_{aik}.
\end{equation}
Therefore $Q^c(F_1,F_1)=0$. Also \eqref{eq:Ki-final} gives $K_1F_2+K_2F_1=8\mu F_1$. Set $(i,k)=(1,2)$ and $(j,l)=(3,p)$ in \eqref{eq:Qc-inversion}. It follows that
\[
(Q^c)_{212}=4\mu(H_0)_{11},\qquad
(Q^c)_{112}=-4\mu(H_0)_{21},\qquad (Q^c)_{312}=0.
\]
Since $(H_0)_{11}>0$, we have
\begin{equation}\label{eq:root-transverse}
Q^c(F_1,F_1)=0,\qquad Q^c(F_1,F_2)\ne0\quad\hbox{on }U_\mu:=\{\mu\ne0\}.
\end{equation}

\step{A second evaluation of the same component.}
The conjugate of \eqref{eq:Q-transport} acts only on the output of $Q^c$. Differentiate the first identity of \eqref{eq:root-transverse} in direction $\bar N$. Its output term vanishes, so
\[
0=2Q^c(\nabla_{\bar N}F_1,F_1).
\]
By \eqref{eq:full-blocks}, $q(\bar N)\mathbf e_1=0$. Thus $\nabla_{\bar N}F_1=(\Pi_{\bar N})_{11}F_1+(\Pi_{\bar N})_{21}F_2$. The second part of \eqref{eq:root-transverse} gives
\[
(\Pi_{\bar N})_{21}=0\quad\hbox{throughout }U_\mu,
\qquad F_2((\Pi_{\bar N})_{21})=0\quad\hbox{there}.
\]
This differentiates an identity on an open set, not an equality at a single point.

Finally evaluate the plane curvature $d\Pi+\Pi\wedge\Pi+p\wedge q$ on $(F_2,\bar N)$. Its $21$ entry is
\begin{align*}
(K_2)_{21}={}&F_2((\Pi_{\bar N})_{21})-\bar N((\Pi_2)_{21})
+[\Pi_2,\Pi_{\bar N}]_{21}\\
&+(p(F_2)q(\bar N)-p(\bar N)q(F_2))_{21}
-a_2(\Pi_{\bar N})_{21}\\
&+\sum_{p=1}^2(\Pi_{\bar N})_{p2}(\Pi_p)_{21}-4\mu(\Pi_N)_{21}.
\end{align*}
The matrices $\Pi_1,\Pi_2$ are diagonal, the off-diagonal product has zero first column, and $(\Pi_N)_{21}=1$. Hence
\begin{equation}\label{eq:K21-second}
(K_2)_{21}=F_2((\Pi_{\bar N})_{21})+(2-a_2)(\Pi_{\bar N})_{21}-4\mu.
\end{equation}
On $U_\mu$, this is $-4\mu$, whereas \eqref{eq:K21-first} gives $-2\mu$ in the same frame and normalization. Thus $U_\mu$ is empty. Proposition~\ref{prop:normalization} makes it open and dense in the nonempty set $U$, a contradiction. This proves the proposition. No normalized field or inverse determinant has been extended beyond $U$.
\end{proof}

\vspace{0.3cm}

\section{The case \texorpdfstring{$\det\B\equiv0$}{det B identically zero}}\label{sec:singular}
Suppose throughout this section that $R\not\equiv0$ and $\delta=\det\B\equiv0$. We retain $O_2,P,O_*$ from Section~\ref{sec:geometry}. For a unit normal $n$ to $P$, write $\alpha_{\rm im}=(\cdot,n)$ and $\tau=\tau_n$.

\vspace{0.3cm}

\noindent {\bf 1). A holomorphic plane and a global one-form.}

\vspace{0.1cm}

\begin{lemma}\label{lem:singular-factor}
On $O_*$, $\B$ has rank two and there is a unique ordinary vector field $U$ such that $\B=w\odot U$. The line $L=\Span(w)\otimes\KM^{-1}$ is preserved by all Chern derivatives of type $(1,0)$. Its orthogonal plane $H=L^\perp$ is holomorphic, and
\begin{equation}\label{eq:H-skew}
\pi_L\nabla_XY+\pi_L\nabla_YX=0,\qquad X,Y\in H.
\end{equation}
There is a global real analytic $(1,0)$-form $\beta$, agreeing with $\iota_UW$ on $O_*$, such that
\begin{equation}\label{eq:singular-beta}
\beta\otimes\beta=-4\adj\B,\quad \partial\beta=-W,
\quad |\beta|^2\leq2t,
\quad\beta(\B(\xi,\cdot))=0,\quad\beta\wedge W=0.
\end{equation}
In particular $\vartheta:=\eta+\beta$ satisfies $\partial\vartheta=0$.
\end{lemma}
\begin{proof}
By Lemma~\ref{lem:alignment}, $\tau(w)=0$. Choose a local unitary frame with $e_1$ along $w$, $e_1,e_2$ spanning $P$, and $e_3=n$. Then
\begin{equation}\label{eq:singular-unitary-B}
\B=\begin{pmatrix}A_0&D&b\\D&0&0\\b&0&0\end{pmatrix}\otimes\nu_0,
\qquad w=s_0e_1\otimes\nu_0,
\qquad bs_0\ne0.
\end{equation}
Indeed, alignment gives $\B_{23}=0$, $\tau\ne0$ gives $b\ne0$, and $\det\B=-b^2\B_{22}=0$ gives $\B_{22}=0$. The nonzero minor $-b^2$ proves rank two. The factorization is
\[
U=(A_0e_1+2De_2+2be_3)/s_0.
\]
The plane recurrence and the divergence of $\Nt$ give $(\Nt_3)_{a3}=(-s_0,0,0)_a$. Since
\[
\adj\B=\begin{pmatrix}0&0&0\\0&-b^2&bD\\0&bD&-D^2\end{pmatrix},
\]
its contraction with the normal derivative of $\B$ gives $(\Nt_3)_{22}=0$. Thus
\begin{equation}\label{eq:singular-normal-N}
\Nt_3=\begin{pmatrix}r_0&v_0&-s_0\\v_0&0&0\\-s_0&0&0\end{pmatrix}\otimes\nu_0.
\end{equation}
Every unbarred derivative of $\B$ therefore has the common factor $L$. Locally write $\B=l\odot V_0$, with $l$ a unit section of $L$ and $V_0\bmod L\ne0$. Differentiating and projecting twice to $E/L$ gives
\[
\pi(\nabla_Xl)\odot\pi(V_0)=0.
\]
Symmetric multiplication by a nonzero vector is injective, so $\nabla_Xl\in L$. Metric compatibility gives $\nabla''H\subset H$. The graph argument in the proof of Proposition~\ref{prop:Qzero} shows that $H$ is an ordinary holomorphic subbundle.

To prove \eqref{eq:H-skew}, write $\Nt_p=l\odot Z_p$, including its free derivative slot. In the frame above, $Z_2=s_0e_2$. On diagonal variables, the left side of \eqref{eq:N-mixed} is
\[
J_l(y)Z_p(y)+y_1K_p(y),\qquad
J_l(y)=\sum_{a,c}(\nabla_{\bar e_c}l,e_a)y_ay_c,
\]
where $K_p$ includes the derivatives of $Z_p$ and of the free $p$-slot. The right side is divisible by $y_1$. Set $y_1=0$ and $p=2$; cancellation gives $J_l(0,y_2,y_3)=0$. Its three coefficients, by metric compatibility, are exactly \eqref{eq:H-skew}.

On $O_*$, $W=s_0\varphi_2\wedge\varphi_3$ and
\[
\beta=\iota_UW=-2b\varphi_2+2D\varphi_3.
\]
The adjugate above gives all the algebraic identities in \eqref{eq:singular-beta}, including $|\beta|^2=4(|b|^2+|D|^2)\leq2|\B|^2$. The adjugate is an ordinary symmetric covariant tensor, because its line weight $(\det E)^2\otimes\KM^2$ is trivial.

Writing $w=u\otimes\nu_0$, alternating-symbol inversion gives
\[
T_0(X,Y)=\tfrac12\{\nu_0(u,X,Y)U+\nu_0(U,X,Y)u\}.
\]
Hence $T_0(X,u)=\beta(X)u/2$ and $T_0(X,U)=-\beta(X)U/2$. Both $\nabla'$ and $D'=\nabla'-C$ preserve $L$, and both have zero pure curvature there: for $D'$, the compatibility endomorphism annihilates $P\supset L$. Their difference on $L$ is $-(\eta+\beta)/2$. Therefore $\partial(\eta+\beta)=0$ on $O_*$.

It remains to extend $\beta$. We use the following elementary analytic removal fact. If real analytic germs $A,C,F$, with $C$ not the zero germ, satisfy $A^2=FC^2$, then $A/C$ extends real analytically. Complexify the convergent power series. The quotient is holomorphic off $C=0$ and locally bounded there, since $|A/C|^2=|F|$. Choose a coordinate direction in which the first nonzero homogeneous part of $C$ is nonzero. On a small circle in that direction, $C$ is nonzero, also for nearby values of the other coordinates. The Cauchy integral over that fixed circle defines a jointly holomorphic extension inside. On each one-variable slice, boundedness removes every possible pole, so this extension agrees with $A/C$. Restricting to the real locus proves the assertion.

At any point, some component $u_j$ of $w$ in a holomorphic coordinate and canonical frame is not the zero analytic germ, because $W$ has no open zero set. On $\{u_j\ne0\}\cap O_*$ the factorization gives
\[
U^j=\B_{jj}/u_j,\qquad
U^i=(2\B_{ij}u_j-\B_{jj}u_i)/u_j^2\quad(i\ne j).
\]
Each coefficient of $\beta=\nu_0(u,U,\cdot)$ is a quotient of analytic functions, and its square is a diagonal coefficient of $-4\adj\B$. The removal fact extends it across its denominator zeros. Local extensions agree on dense overlaps and glue. All equations in \eqref{eq:singular-beta} then extend by continuity. This extends only the ordinary form $\beta$, not the vector $U$.
\end{proof}

\vspace{0.3cm}

\noindent {\bf 2). A line bundle of negative weighted curvature.}

\vspace{0.2cm}

Let $\mathscr A$ be the smooth trivial line equipped with the barred operator
\begin{equation}\label{eq:shift-line}
\bp_{\mathscr A}=\bp-\bar\vartheta/2.
\end{equation}
It is integrable because $\partial\vartheta=0$. To describe its local holomorphic frames, solve $\bp g=\bar\vartheta/2$ on a small polydisk and take $e^g$ times the global smooth frame. Such a solution follows from successive one-variable Cauchy--Green operators. The frames can be chosen real analytic: complexify the analytic closed barred form and use the radial Poincar\'e integral in the independent barred variables, then restrict to the conjugate diagonal. Give the global smooth frame norm one. Its Chern connection and real curvature are
\begin{equation}\label{eq:shift-curvature}
\nabla^{\mathscr A}=d+(\vartheta-\bar\vartheta)/2,
\qquad \rho_{\mathscr A}=\frac{\sqrt{-1}}2(\bp\vartheta-\partial\bar\vartheta).
\end{equation}

\begin{lemma}\label{lem:negative-line}
Let $v$ and $\Omega=v\omega^2/2$ be as in Lemma~\ref{lem:adjoint-density}. Then
\[
\int_M\rho_{\mathscr A}\wedge\Omega<0,
\qquad H^0(M,\mathscr A^2)=0.
\]
\end{lemma}
\begin{proof}
Put $f=\frac12\log v$ and $\eta_f=\eta-2\partial f$. By Lemma~\ref{lem:cubic-dense}, $S$ is nonzero somewhere. Thus $\psi=\chi-2t\geq0$ by Proposition~\ref{prop:compact-estimate}; moreover $\psi\not\equiv0$ by \eqref{eq:scalar}. Define
\[
A_0=\int_Mv|\eta_f|^2\,dV,\qquad B_0=\int_Mv|\beta|^2\,dV.
\]
Equations~\eqref{eq:weighted-trace} and \eqref{eq:singular-beta} give
\begin{equation}\label{eq:weighted-strict-gap}
A_0-B_0=\int_Mv(\chi-|\beta|^2)\,dV\geq\int_Mv\psi\,dV>0.
\end{equation}

Set $I=\int_M\bar\vartheta\wedge\partial\Omega$. Stokes' theorem gives $\int_M\partial\bar\vartheta\wedge\Omega=I$ and $\int_M\bp\vartheta\wedge\Omega=\bar I$. Hence $\int_M\rho_{\mathscr A}\wedge\Omega=-\re(\sqrt{-1}I)$. Replacing $\vartheta$ by $\vartheta_f=\vartheta-2\partial f$ leaves $I$ unchanged, since $\partial\bp\Omega=0$. Using $\partial\Omega=-\eta_f\wedge\Omega$ and the unitary identity
\[
\sqrt{-1}\,\xi\wedge\bar\zeta\wedge\omega^2/2=(\xi,\zeta)dV,
\]
we get
\[
\re(\sqrt{-1}I)=\int_Mv\re(\eta_f,\vartheta_f)\,dV
=A_0+\re\int_Mv(\eta_f,\beta)\,dV.
\]
Cauchy--Schwarz and \eqref{eq:weighted-strict-gap} give
\begin{equation}\label{eq:negative-degree}
\int_M\rho_{\mathscr A}\wedge\Omega
\leq-A_0+\sqrt{A_0B_0}
\leq-\tfrac12(A_0-B_0)<0.
\end{equation}

We also give the section argument, including its zeros. For any Hermitian holomorphic line bundle $F$ with real curvature $\rho_F$, let $k_F$ be its mean curvature. If $s$ is a nonzero holomorphic section and $u=|s|^2$, the Chern product rule gives
\[
\Delta u=|\nabla's|^2-k_Fu,\qquad
|\partial u|^2=u|\nabla's|^2.
\]
For each $\epsilon>0$,
\[
\Delta\log(u+\epsilon)
=\frac{\epsilon|\nabla's|^2}{(u+\epsilon)^2}-\frac{k_Fu}{u+\epsilon}.
\]
Integrate against $v$ and use \eqref{eq:adjoint-density}. It follows that $\int_Mvk_Fu/(u+\epsilon)\,dV\geq0$. The zero set of a nonzero holomorphic section has measure zero: this follows locally by slicing its holomorphic coefficient in a coordinate with a first nonzero Taylor derivative, and applying induction and Fubini to the exceptional slices. Dominated convergence therefore gives $\int_Mvk_F\,dV\geq0$, or equivalently $\int_M\rho_F\wedge\Omega\geq0$. Applying this to $F=\mathscr A^2$ contradicts \eqref{eq:negative-degree}. Thus $H^0(M,\mathscr A^2)=0$.
\end{proof}

\par
\Needspace{6\baselineskip}
\vspace{0.3cm}
\noindent\textbf{3). The holomorphic quotient and the restricted quadratic form.}\par
\nopagebreak[4]
\vspace{0.1cm}
\nopagebreak[4]

\begin{proposition}\label{prop:singular-conic}
There is a holomorphic line bundle $\mathscr Q$ and a holomorphic map $\alpha_0:E\longrightarrow\mathscr Q$ whose zero set $\Sigma$ has complex codimension at least two and is disjoint from $O_*$. On $M\setminus\Sigma$, its kernel is the holomorphic plane $H$ extending the plane in Lemma~\ref{lem:singular-factor}. The sections
\begin{equation}\label{eq:global-skew-sections}
\mathfrak a=\alpha_0\wedge(\alpha_0\circ T),\qquad
\mathfrak l=-\alpha_0\wedge\partial\alpha_0,\qquad
\mathfrak k=(\mathfrak a+\mathfrak l)/2
\end{equation}
of $\mathscr Q^2\otimes\KM$ are globally real analytic, and $\mathfrak l$ is holomorphic. They represent, respectively, the normal torsion, the Levi bracket, and the skew second fundamental tensor on $H$.

The restriction of $Q^c=\bar Q$ to $\Sym^2H$ is a holomorphic binary quadratic form with values in
\[
\mathscr J=\mathscr Q^{-1}\otimes\KM^{-1}\otimes\mathscr A
\]
under the induced output identification, and its determinant is zero. Consequently, in the unitary frame of \eqref{eq:singular-unitary-B}, the three quadratic rows of $Q$ have the form
\begin{equation}\label{eq:singular-Q-rows}
Q_1=\begin{pmatrix}u_{11}&u_{12}&u_{13}\\u_{12}&u_{22}&u_{23}\\u_{13}&u_{23}&u_{33}\end{pmatrix},\quad
Q_2=\begin{pmatrix}v_{11}&v_{12}&v_{13}\\v_{12}&0&0\\v_{13}&0&0\end{pmatrix},\quad Q_3=0,
\qquad u_{22}u_{33}=u_{23}^2.
\end{equation}
The coefficient $u_{23}$ is invariant under unit phase changes of the three adapted vectors, and its nonzero set is open and dense in $O_*$.
\end{proposition}
\begin{proof}
\step{The primitive holomorphic quotient.}
On a small coordinate ball write $w=u\otimes\nu_0$ and $a_0=(\cdot,u)$. This is an analytic covector, whose coefficient ratios are holomorphic on $O_*$ because $\ker a_0=H$ is holomorphic there. Hence
\[
(a_0)_i\partial_{\bar z}(a_0)_j-(a_0)_j\partial_{\bar z}(a_0)_i=0
\]
throughout the ball by analyticity. Choose a coefficient not equal to the zero germ. Complexifying a numerator and denominator gives holomorphic functions $P(z,\zeta),C(z,\zeta)$ satisfying $C\partial_\zeta P-P\partial_\zeta C=0$. Choose $\zeta_0$ such that $C(z,\zeta_0)\not\equiv0$. The quotient is independent of $\zeta$ where defined, so
\[
P(z,\zeta)C(z,\zeta_0)=P(z,\zeta_0)C(z,\zeta).
\]
The holomorphic identity principle extends this to the product polydisk. Restricting to $\zeta=\bar z$ shows that the original ratio is the meromorphic germ $P(z,\zeta_0)/C(z,\zeta_0)$. Thus $\Ann(H)$ defines a meromorphic line in $E^*$.

The local ring of holomorphic germs on a smooth complex chart is factorial, and a pure codimension-one analytic germ has a principal vanishing ideal; see \cite[Theorems 6.4.2 and 6.6.1]{Lebl}. Clear the denominators of a representative of the meromorphic line and remove all common irreducible factors. Its coefficients are then holomorphic and have no common divisorial component, so their common zero set has codimension at least two. Two such primitive representatives differ by a holomorphic unit: the intervening meromorphic scalar has valuation zero at every prime divisor, and in a reduced quotient neither numerator nor denominator can have a nonunit factor. The representatives therefore glue as a holomorphic line subsheaf of $E^*$. Its dual is a line bundle $\mathscr Q$, and its dual map is $\alpha_0:E\longrightarrow\mathscr Q$.

Let $\Sigma=\{\alpha_0=0\}$. On $O_*$, the annihilator has a local nowhere-zero holomorphic representative. A primitive representative differs from it by a holomorphic scalar, and a zero of that scalar would be a common divisorial factor. Thus $\Sigma\cap O_*=\varnothing$. On $M\setminus\Sigma$, the kernel of $\alpha_0$ is a holomorphic plane agreeing with $H$ on $O_*$. Its orthogonal line and plane are smooth and analytic. Preservation and skewness extend there by density.

Under a holomorphic change of frame of $\mathscr Q$, the coefficients of $\alpha_0\wedge\partial\alpha_0$ transform as a section of $\mathscr Q^2\otimes\KM$: the extra differentiated term contains $\alpha_0\wedge\alpha_0=0$. This proves the assertions about \eqref{eq:global-skew-sections}. On $H$ inputs $X,Y$, divide their evaluation on $(Z,X,Y)$ by $\alpha_0(Z)$ for any transverse $Z$. The resulting values are $\alpha_0(T(X,Y))$, $\alpha_0([X,Y])$, and $\alpha_0(\nabla_XY)$, respectively. The last uses \eqref{eq:H-skew} and the torsion identity.

\step{The actual output line.}
Put $V_{\mathrm{out}}=\overline{E\otimes\KM}$. The metric gives parallel complex-linear maps $E^*\to\bar E$ and $\KM^{-1}\to\bar\KM$. Apply them to the holomorphic dual map $\alpha_0^\vee:\mathscr Q^{-1}\to E^*$ to obtain
\[
j_0:\mathscr Q^{-1}\otimes\KM^{-1}\longrightarrow V_{\mathrm{out}}.
\]
Its zero set is exactly $\Sigma$, and off $\Sigma$ its image is $\overline{L\otimes\KM}$. It intertwines the original barred Chern operators. Define
\[
\mathcal D''=\nabla''_{V_{\mathrm{out}}}-\bar C.
\]
By the proof of Lemma~\ref{lem:singular-factor}, $C$ acts on $L$ as $\vartheta/2$. Tensoring $j_0$ with the global unit smooth frame of $\mathscr A$ gives a map $j:\mathscr J\to V_{\mathrm{out}}$ satisfying
\[
\mathcal D''j=j\bp_{\mathscr J}.
\]
This holds first off $\Sigma$ and then everywhere by smoothness. No integrability of $\mathcal D''$ on the whole output bundle is needed.

We claim that $Q^c(\Sym^2H)$ takes values in $\operatorname{Im}j$. The second fundamental tensor $A(X,Y)=\alpha_0(\nabla_XY)$ for $X,Y\in H$ is skew. The mixed curvature formula gives
\[
\alpha_0(R(X,\bar Z)Y)=-(D_{\bar Z}A)(X,Y).
\]
To check the derivative slots, the $(1,0)$ component of $[X,\bar Z]$ is $-\nabla_{\bar Z}X$, since the mixed torsion is zero; it supplies exactly the derivative of the first input of $A$. The quotient connection supplies the output derivative. Thus
\[
\alpha_0(R(X,\bar Z)Y)+\alpha_0(R(Y,\bar Z)X)=0.
\]
By curvature reality, for a unit section $l$ of $L$ this is
\[
R(Z,\bar X,l,\bar Y)+R(Z,\bar Y,l,\bar X)=0.
\]
The barred-symmetric curvature is alternating in its unbarred pair. Its associated two-form therefore vanishes whenever one unbarred input is $l$, so its alternating-symbol image lies in $L\otimes\KM$. This proves the claim.

There is now a unique symmetric form $C_H$ on $H$ with values in $\mathscr J$ such that $Q^c|_{\Sym^2H}=jC_H$. Conjugating \eqref{eq:Q-transport} gives $\mathcal D''Q^c=0$, with the ordinary Dolbeault operator on the holomorphic source. Since $H$ is holomorphic and $j$ intertwines, $\bp C_H=0$ off $\Sigma$. The exact sequence $0\to H\to E\to\mathscr Q\to0$ there gives the holomorphic line isomorphism
\begin{equation}\label{eq:conic-det-line}
(\det H^*)^2\otimes\mathscr J^2
=(\KM\otimes\mathscr Q)^2\otimes
(\mathscr Q^{-1}\otimes\KM^{-1}\otimes\mathscr A)^2
\simeq\mathscr A^2.
\end{equation}

\step{Extension of the determinant.}
In holomorphic frames write $\alpha_0=(a_i)$. Where $a_i\ne0$, the vectors $E_j-(a_j/a_i)E_i$, $j\ne i$, frame $H$. Evaluating $Q^c$ on these vectors and dividing by a nonzero component of $j$ expresses the coefficients of $C_H$ as quotients of real analytic functions. Use the real analytic holomorphic frames of $\mathscr A$ described after \eqref{eq:shift-line}. Then \eqref{eq:conic-det-line} expresses $\det C_H$ as a real analytic quotient in a holomorphic frame of $\mathscr A^2$. All denominators can be chosen as nonzero germs. The coefficient is holomorphic off $\Sigma$, so the same complexification argument used for the annihilator expresses it as a meromorphic holomorphic germ across $\Sigma$.

Such a germ has no pole. Indeed, in a reduced quotient in the factorial local ring, any irreducible denominator factor defines a hypersurface with a dense open part outside the codimension-at-least-two set $\Sigma$. Holomorphicity there makes the numerator vanish on that hypersurface. The principal vanishing-ideal statement then makes the same factor divide the numerator, a contradiction. Thus the denominator is a unit. The local extensions glue to a global holomorphic section of $\mathscr A^2$, which is zero by Lemma~\ref{lem:negative-line}. Hence $\det C_H=0$.

\step{The adapted rows and the nonzero middle coefficient.}
On $O_*$, $H=\Span\{e_2,e_3\}$. The image of $Q$ lies in $P\otimes\KM$, so $Q_3=0$. The restricted output line makes the $22,23,33$ entries of $Q_2$ zero. The determinant just proved gives $u_{22}u_{33}=u_{23}^2$. This is \eqref{eq:singular-Q-rows}.

For use below, compute from \eqref{eq:Q-decomp} that
\begin{equation}\label{eq:singular-G}
G=2\begin{pmatrix}
v_{13}&0&0\\-u_{13}&-u_{23}&-u_{33}\\
u_{12}-v_{11}&u_{22}-v_{12}&u_{23}-v_{13}
\end{pmatrix}.
\end{equation}
The nonzero independent cubic entries can only be
\[
\begin{gathered}
S_{111}=-2u_{11},\quad S_{112}=-2(2u_{12}+v_{11})/3,\quad S_{113}=-4u_{13}/3,\\
S_{122}=-2(u_{22}+2v_{12})/3,\quad S_{123}=-2(u_{23}+v_{13})/3,\quad S_{133}=-2u_{33}/3.
\end{gathered}
\]
Thus
\begin{align}
6|S|^2={}&24|u_{11}|^2+8|2u_{12}+v_{11}|^2+32|u_{13}|^2\notag\\
&+8|u_{22}+2v_{12}|^2+16|u_{23}+v_{13}|^2+8|u_{33}|^2,\label{eq:singular-S-norm}\\
\tr G^2={}&8(v_{13}^2-u_{23}v_{13}+u_{33}v_{12}).\label{eq:singular-trG}
\end{align}
Put $a_0=\re u_{23}$ and $b_0=\im u_{23}$. Completing squares, using $u_{22}u_{33}=u_{23}^2$, gives
\begin{align}
r_G-6|S|^2={}&30a_0^2-6b_0^2-\tfrac{15}{2}|u_{33}|^2
-24|u_{11}|^2-8|2u_{12}+v_{11}|^2-32|u_{13}|^2\notag\\
&-32|v_{12}+u_{22}/2-\bar u_{33}/8|^2
-8(\re v_{13}+5a_0/2)^2\notag\\
&-24(\im v_{13}+b_0/2)^2.
\label{eq:middle-square}
\end{align}
For example, the $v_{13}$ completion leaves $34a_0^2-10b_0^2$, and the $v_{12}$ completion leaves $|u_{33}|^2/2-4\re(u_{22}u_{33})$; together these give the first three terms displayed.

Under unit phases $e_i\mapsto z_ie_i$ and the matching canonical frame, $Q_{alj}$ acquires the factor $z_1z_2z_3/(z_az_lz_j)$. Thus $u_{23}=Q_{123}$ is phase invariant. To continue only its zero condition, use \eqref{eq:global-detectors} and put
\[
A_w=ww^*,\qquad Z=\sigma|w|^2I-|w|^2J_Q-\sigma A_w.
\]
On $O_*$, $Z=\sigma|w|^2\pi_{\Span(e_2)}$. If $\widehat G$ is the endomorphism corresponding to the $(1,1)$-tensor $G$, then the global analytic scalar
\[
E_G=-\tfrac12\tr(Z\widehat G)
\]
equals $\sigma|w|^2u_{23}$ on $O_*$. If $u_{23}$ vanished on an open patch, analytic continuation would make $E_G=0$ globally and hence $u_{23}=0$ throughout $O_*$. Equation~\eqref{eq:middle-square} would give $r_G\leq6|S|^2$ on $O_*$, and then on $M$ by continuity. Multiplication by $t$ and \eqref{eq:compact-estimate} would contradict $\int_Mt^3>0$. Thus $u_{23}$ has no nonempty open zero set.
\end{proof}

\par
\Needspace{6\baselineskip}
\vspace{0.3cm}
\noindent\textbf{4). The full recurrence on the regular set.}\par
\nopagebreak[4]
\vspace{0.1cm}
\nopagebreak[4]

\begin{proposition}\label{prop:singular-recurrence}
On $O_*$ there is an intrinsic one-form $\zeta$, characterized by
\[
w=q\B(\alpha_{\rm im},\cdot),\qquad \zeta=q\alpha_{\rm im},
\]
such that
\begin{equation}\label{eq:singular-B-recurrence}
\partial\zeta=W,\qquad
\nabla_X\B=(\eta(X)/2-3\zeta(X)/2)\B+w\odot X.
\end{equation}
There are unique smooth ordinary vectors $U,V$ and a scalar $\nu_s$ with
\begin{equation}\label{eq:singular-factors}
\B=w\odot U,\qquad S=\B\odot V,\qquad \Sym_2\nabla''w=\nu_s\B.
\end{equation}
Put $\xi=-\iota_UW/2$, $a=\zeta/2-\xi$, and $r=\eta/2+\xi$. Then
\begin{equation}\label{eq:singular-basic-recurrences}
\begin{gathered}
\zeta(U)=2,\qquad\xi(U)=0,\qquad \partial\xi=W/2,\qquad\partial a=0,\\
\nabla_Xw=(\eta-3\zeta/2+\xi)(X)w,\qquad
\nabla_XU=X-r(X)U,\\
(\nabla_{\bar j}\zeta)_p+\zeta_pV^j=(G^*)_{pj}/3+2\nu_s\delta_{pj}/3.
\end{gathered}
\end{equation}
There is a unique smooth scalar $\gamma$ such that, for $Y=V+\gamma U$,
\begin{equation}\label{eq:singular-VY}
\nabla_XV=a(X)V+\gamma(\zeta(X)U-X),\qquad
\partial\gamma=\tfrac12\gamma(\eta-\zeta),\qquad \nabla_XY=a(X)Y.
\end{equation}
The vectors $U,V$ are independent at every point, and $w\wedge U\wedge Y$ is nonzero on an open dense subset of $O_*$. In the frame of \eqref{eq:singular-unitary-B}, if
\[
\ell(y)=A_0y_1+2Dy_2+2by_3,
\]
then the quadratic rows are
\begin{equation}\label{eq:singular-row-factor}
Q_3=0,\qquad Q_2(y)=y_1M(y),\qquad
M(y)=\frac{s_0}{2b}(\gamma U(y)-V(y)),\qquad
Q_1(y)=-\frac12\ell(y)V(y)-y_2M(y).
\end{equation}
\end{proposition}
\begin{proof}
\step{The mixed equation and the skew tensor.}
Retain the unitary frame in \eqref{eq:singular-unitary-B}. By \eqref{eq:universal-remainder} and \eqref{eq:singular-normal-N},
\[
\Nt_X=w\odot X+\alpha_{\rm im}(X)C_0,\qquad
C_0=\begin{pmatrix}F&H_0&-3s_0/2\\H_0&0&0\\-3s_0/2&0&0\end{pmatrix}\otimes\nu_0.
\]
Set
\[
\ell_B=A_0y_1+2Dy_2+2by_3,\qquad
\ell_C=Fy_1+2H_0y_2-3s_0y_3,
\]
\[
L_p(y)=\sum_j(\nabla_{\bar j}\alpha_{\rm im})_p y_j,\qquad
K_p(y)=\sum_j(G^*)_{pj}y_j,\qquad P_w(y)=\Sym_2\nabla''w(y,y).
\]
Differentiating the complete recurrence, including its free $p$-slot, gives
\begin{equation}\label{eq:singular-mixed1}
y_pP_w+y_1\ell_CL_p=-\frac12y_1\ell_BK_p,
\qquad p=1,2.
\end{equation}
The $p=2$ equation gives $P_w=y_1L_w$. Cancel $y_1$ and eliminate $L_w$ to obtain
\begin{equation}\label{eq:singular-mixed2}
\ell_C(y_2L_1-y_1L_2)=-\frac12\ell_B(y_2K_1-y_1K_2).
\end{equation}
If $\ell_B,\ell_C$ are independent, linear divisibility gives
\[
y_2L_1-y_1L_2=\ell_B(c_1y_1+c_2y_2+c_3y_3).
\]
The $y_3^2$ coefficient gives $c_3=0$. Put $\kappa=(\nabla_{e_2}e_3,e_1)$. By \eqref{eq:H-skew}, $(\nabla_{e_3}e_3,e_1)=0$, so metric compatibility makes the $y_3,y_2$ coefficients of $L_1$ equal to $0,\bar\kappa$, respectively. The $y_2y_3$ coefficient gives $c_2=0$, and then the $y_2^2$ coefficient gives $\kappa=0$.

Consequently, wherever the global skew section $\mathfrak k$ in \eqref{eq:global-skew-sections} is nonzero, $\ell_C$ is proportional to $\ell_B$. Their $y_3$ coefficients give
\begin{equation}\label{eq:singular-proportionality}
C_0=-\frac{3s_0}{2b}\B.
\end{equation}
We must first exclude $\mathfrak k\equiv0$ before using this on a dense set.

\step{Consequences of a zero skew tensor.}
Suppose $\mathfrak k=0$. On $M\setminus\Sigma$, let $A(X,Z)=\alpha_0(\nabla_XZ)$ for $Z\in H$ be the full second fundamental tensor. Its $H$-input restriction is zero; denote its remaining $L$-input restriction by $r_L$. Evaluate the zero off-diagonal $(2,0)$ curvature on $X,Y,Z\in H$. The differentiated $H$-input terms vanish, as do the terms with $\nabla_XY,\nabla_YX\in H$. The remaining normal bracket term is $-\mathfrak l(X,Y)r_L(Z)$, so $\mathfrak l\otimes r_L=0$.

If $\mathfrak l\ne0$ on an open set, that set meets $O_*$, and there $r_L=0$. Thus the whole second fundamental tensor vanishes, and the splitting $L\oplus H$ is parallel in both types. Put $f_{i\bar j}=R_{i\bar j1\bar1}$ in an adapted frame. All $R_{i\bar j1\bar a}$ and $R_{i\bar ja\bar1}$ vanish for $a=2,3$. The constant and linear coefficients in $H^c(e_1+ze_a)=0$ give $f_{1\bar1}=f_{1\bar a}=f_{a\bar1}=0$. The $12,13$ barred entries of $Q_3=0$ give $f_{2\bar2}=f_{2\bar3}=0$, and reality gives $f_{3\bar2}=0$.

Since $\B$ has the parallel line $L$ as a factor, $(\nabla_{\bar1}\B)_{23}=0$. The mixed Bianchi identity gives
\[
2(\nabla_{\bar1}\B)_{23}
=R_{1\bar13\bar3}-R_{3\bar11\bar3}
+R_{2\bar11\bar2}-R_{1\bar12\bar2}.
\]
The middle terms vanish. The $|z|^2$ coefficient of $H^c(e_1+ze_a)=0$ gives $R_{1\bar1a\bar a}+f_{a\bar a}=0$, so $f_{3\bar3}=f_{2\bar2}=0$. All $f$ vanish. The formula $Q_{2jl}=(R_{3\bar j1\bar l}+R_{3\bar l1\bar j})/2$ then gives $Q_2=0$, contrary to rank two. Therefore $\mathfrak l=0$ globally, and $\mathfrak a=2\mathfrak k-\mathfrak l=0$. This conclusion does not make $r_L$ vanish when $\mathfrak l=0$.

\step{Exclusion of the zero-skew case.}
The equality $\mathfrak a=0$ gives $A_0=0$ in \eqref{eq:singular-unitary-B}. Write $\Gamma_p{}^a_b=(\nabla_{e_p}e_b,e_a)$. Preservation of $L$ gives $\Gamma_p{}^2_1=\Gamma_p{}^3_1=0$ for every $p$, while $\mathfrak k=0$ gives $\Gamma_p{}^1_b=0$ for $p,b\in\{2,3\}$. Retain
\[
r_2=\Gamma_1{}^1_2,\qquad r_3=\Gamma_1{}^1_3.
\]
Since $\B_{11}=0$ and $L$ is unbarred-preserved, metric compatibility gives $(\nabla_{\bar j}\B)_{11}=0$; hence $u_{11}=0$. In \eqref{eq:singular-mixed1}, $L_1=\bar r_3y_1$ and $P_w=y_1L_w$, so
\[
y_1(L_w+\bar r_3\ell_C)=-\frac12\ell_BK_1.
\]
Here $\ell_B=2Dy_2+2by_3$ has no $y_1$ factor. Therefore $y_1$ divides $K_1$, giving $G_{21}=G_{31}=0$. By \eqref{eq:singular-G}, $u_{13}=0$ and $u_{12}=v_{11}$.

Compare the full covariant derivative of $G$ with \eqref{eq:G-transport}. At $(p,r,j)=(3,2,1)$ the former is zero by the displayed connection coefficients, while the latter is $-2bv_{11}-2Du_{13}$. Thus $v_{11}=u_{12}=0$. We now have
\[
G=2\begin{pmatrix}v_{13}&0&0\\0&-u_{23}&-u_{33}\\0&u_{22}-v_{12}&u_{23}-v_{13}\end{pmatrix},
\qquad Q(\bar e_1,\bar e_1)=0.
\]
At $(p,r,j)=(1,2,1),(1,3,1)$, the same two derivative formulas give
\begin{equation}\label{eq:zero-skew-system}
(v_{13}+u_{23})r_2+u_{33}r_3=bv_{13},\qquad
(u_{22}-v_{12})r_2+(u_{23}-2v_{13})r_3=bv_{12}.
\end{equation}
Differentiating the actual zero root $Q(\bar e_1,\bar e_1)=0$ along $e_1$ gives $v_{12}r_2+v_{13}r_3=0$, since the barred input derivative is a multiple of $\bar e_1$ minus $r_2\bar e_2+r_3\bar e_3$. The three equations form a homogeneous system for the nonzero vector $(r_2,r_3,b)$, so
\begin{equation}\label{eq:zero-skew-determinant}
u_{22}v_{13}^2-2u_{23}v_{12}v_{13}+u_{33}v_{12}^2=0.
\end{equation}
The surviving entries in \eqref{eq:singular-S-norm} give
\[
6|S|^2=8|u_{22}+2v_{12}|^2+16|u_{23}+v_{13}|^2+8|u_{33}|^2,
\qquad r_G=8\re(v_{13}^2-u_{23}v_{13}+u_{33}v_{12}).
\]
If $u_{33}=0$, then $u_{23}=0$, and $r_G\leq8|v_{13}|^2\leq3|S|^2$. If $u_{33}\ne0$, multiply \eqref{eq:zero-skew-determinant} by $u_{33}$ and use $u_{22}u_{33}=u_{23}^2$. This gives $(u_{23}v_{13}-u_{33}v_{12})^2=0$. Put $z=u_{23}+v_{13}$. Then
\[
u_{22}+2v_{12}=(z^2-v_{13}^2)/u_{33},\qquad r_G=8\re(v_{13}^2),
\]
and
\[
6|S|^2\geq16|z^2-v_{13}^2|+16|z|^2\geq16|v_{13}|^2.
\]
Again $r_G\leq3|S|^2$. This bound extends from $O_*$ to $M$ and contradicts \eqref{eq:compact-estimate}, since $\int_Mt^3>0$. Thus $\mathfrak k\not\equiv0$.

Analyticity makes $\{\mathfrak k\ne0\}$ dense. Equation~\eqref{eq:singular-proportionality} therefore extends to all of $O_*$, giving the recurrence in \eqref{eq:singular-B-recurrence} with $\zeta=(s_0/b)\alpha_{\rm im}$. To obtain its exterior equation, use \eqref{eq:normal-equation} and differentiate $w=q\B(\alpha_{\rm im},\cdot)$ along $P$. They give
\[
\partial\alpha_{\rm im}=(k_{\rm im}-\eta/2)\wedge\alpha_{\rm im},
\qquad Xq=(\eta(X)/2+\tau(X)-k_{\rm im}(X))q.
\]
Thus $\partial(q\alpha_{\rm im})=q\tau\wedge\alpha_{\rm im}=W$, using $\tau=b\varphi_2$ and $W=s_0\varphi_2\wedge\varphi_3$. The phase factors cancel, so $\zeta$ is intrinsic.

\step{The factor $U$ and the derivative of $w$.}
The factorization in Lemma~\ref{lem:singular-factor} gives $\zeta(U)=2$ and
\[
\xi=-\tfrac12\iota_UW=b\varphi_2-D\varphi_3,\qquad\xi(U)=0.
\]
Let $a_0=(\eta-3\zeta)/2$. Exterior differentiation of \eqref{eq:singular-B-recurrence} gives
\[
0=-W(X,Y)\B+(\nabla_Xw-a_0(X)w)\odot Y
-(\nabla_Yw-a_0(Y)w)\odot X+w\odot T(X,Y).
\]
After separating the trace torsion, put $H_X=\nabla_Xw-(\eta-3\zeta/2)(X)w$. Then
\begin{equation}\label{eq:singular-w-compat}
0=-W(X,Y)\B+H_X\odot Y-H_Y\odot X+w\odot T_0(X,Y).
\end{equation}
Preservation of $L$ gives $H_X=h_Xw$. The torsion components are
\[
T_0(e_1,e_2)=be_1,\quad T_0(e_3,e_1)=De_1,
\quad T_0(e_2,e_3)=A_0e_1+De_2+be_3.
\]
Evaluate \eqref{eq:singular-w-compat} on $(e_1,e_2)$ and $(e_3,e_1)$. Cancellation of $w$ gives $h_1=0,h_2=b,h_3=-D$, so $h=\xi$. The remaining pair agrees with $\B=w\odot U$. Thus the stated recurrence for $w$ holds. Its pure curvature gives $\partial(\eta-3\zeta/2+\xi)=0$, hence $\partial\xi=W/2$. Differentiating $\B=w\odot U$ and canceling $w$ gives $\nabla_XU=X-(\eta/2+\xi)(X)U$.

\step{Factorization of the cubic and the second vector.}
Write $\B(y)=y_1\ell(y)$ and let $Z_p(y)=\sum_j(\nabla_{\bar j}\zeta)_py_j$. The complete mixed equation for $\Nt_p=w\odot e_p-3\zeta_p\B/2$ is
\begin{equation}\label{eq:singular-mixed-full}
y_pP_w-\frac32\zeta_pS(y)
=\B(y)\bigl(\tfrac32Z_p(y)-\tfrac12K_p(y)\bigr).
\end{equation}
For $p=2$, $\zeta_2=0$, so $y_1\ell$ divides $y_2P_w$. First $y_1$ divides $P_w$. Since $\ell$ has nonzero $y_3$ coefficient and is not proportional to $y_2$, the remaining factor must be proportional to $\ell$. Thus $P_w=\nu_s\B(y)$. The $p=3$ equation, with $\zeta_3=s_0/b\ne0$, gives $S(y)=\B(y)V(y)$ for a unique linear form $V(y)$. Polynomial injectivity and smooth left inverses make $\nu_s,V$ smooth and compatible on overlaps. Substitution in \eqref{eq:singular-mixed-full} proves the last equation in \eqref{eq:singular-basic-recurrences}.

By \eqref{eq:singular-Q-rows}, $Q_3=0$ and $Q_2(y)=y_1M(y)$. Since $S=-2\Sym_3Q$, we get $Q_1(y)=-\ell(y)V(y)/2-y_2M(y)$. Differentiate $S=\B\odot V$ and use \eqref{eq:S-transport}. After canceling the $\eta$-term,
\begin{equation}\label{eq:V-polynomial}
\B(y)\bigl(\nabla_pV(y)-\tfrac32\zeta_pV(y)\bigr)+s_0y_1y_pV(y)
=-2\sum_{u,v}\eps_{puv}(\B_{av}y_a)Q_u(y).
\end{equation}
For $p=1,2,3$, the right sides are, respectively,
\[
-2by_1^2M,\qquad2by_1Q_1,\qquad
-2Dy_1Q_1+2y_1(A_0y_1+Dy_2+by_3)M.
\]
Cancel $y_1$. The $p=2$ identity becomes
\[
\ell(\nabla_2V+bV)=-y_2(s_0V+2bM).
\]
Linear divisibility gives $s_0V+2bM=\gamma\ell$ for a scalar $\gamma$. The other two equations give
\[
\nabla_1V=-\gamma e_1,\quad
\nabla_2V=-bV-\gamma e_2,\quad
\nabla_3V=(D+s_0/(2b))V+(\gamma/b)\ell-\gamma e_3,
\]
where $\ell=s_0U$ under the vector identification. These are the first equation of \eqref{eq:singular-VY} and the row formulas \eqref{eq:singular-row-factor}. The map $X\mapsto\zeta(X)U-X$ has eigenvalues $1,-1,-1$, so $\gamma$ is uniquely determined and smooth. Equivalently, $\gamma=-\tr(\nabla'V-a\otimes V)$.

\step{The recurrent line and its nondegeneracy.}
Since $\partial a=0$, the recurrences for $U,V$ give
\[
\nabla_XY=a(X)Y+\beta_0(X)U,\qquad
\beta_0=\partial\gamma+\tfrac12\gamma(\zeta-\eta).
\]
Pure curvature, with the bracket included, gives
\[
0=[\partial\beta_0-(a+r)\wedge\beta_0](X,Z)U
+\beta_0(Z)X-\beta_0(X)Z.
\]
Modulo $\Span(U)$, choose $X,Z$ with independent quotient classes to obtain $\beta_0(X)=\beta_0(Z)=0$. Then take $X=U$ and $Z$ with nonzero quotient to obtain $\beta_0(U)=0$. Thus $\beta_0=0$, proving the last two equations of \eqref{eq:singular-VY}.

If $V=cU$ at a point, restrict $Q_1$ in \eqref{eq:singular-row-factor} to $y_1=0$. With $\ell_0=Dy_2+by_3$ and $q_0=s_0/b$, it is
\[
-\frac1{s_0}\ell_0\{2c\ell_0+q_0(\gamma-c)y_2\}.
\]
The symmetric-matrix determinant of $(ay_2+by_3)(cy_2+dy_3)$ is $-(ad-bc)^2/4$. The zero restricted determinant and $bq_0\ne0$ therefore force $c=\gamma$. Then $M=0$ and $Q_2=Q_3=0$, contradicting rank two. Hence $U,V$ are independent everywhere.

Suppose $w\wedge U\wedge Y$ vanished on an open set. Since $w,U$ are independent, there $V=cU+le_1$ smoothly. Restricting to $y_1=0$ as above gives $c=\gamma$. Thus $F=V-\gamma U$ lies in $L$, and its derivative is
\[
\nabla_XF=a(X)F+2\gamma(\zeta(X)U-X).
\]
The line $L$ is preserved by all unbarred derivatives. Taking $X=e_2$, for which $\zeta(e_2)=0$, gives $\gamma=0$. Then $V\in L$ and \eqref{eq:singular-row-factor} gives $Q_1|_{y_1=0}=0$ on that open set. In particular $u_{23}=0$ there, contrary to Proposition~\ref{prop:singular-conic}. Thus the nonzero set of $w\wedge U\wedge Y$ is open and dense in $O_*$. The wedge is an ordinary scalar, because its line factor is $\KM\otimes\det E$, and no extension of it beyond $O_*$ is used.
\end{proof}

\vspace{0.3cm}

\noindent {\bf 5). The determinant identity and the final estimate.}

\vspace{0.1cm}

\begin{proposition}\label{prop:singular-flat}
If $\det\B\equiv0$, then $R=0$.
\end{proposition}
\begin{proof}
Suppose $R\not\equiv0$ and use Proposition~\ref{prop:singular-recurrence}. Let
\[
O=\{x\in O_*:w\wedge U\wedge Y\ne0\}.
\]
It is open and dense in $O_*$, and hence dense in $M$. All calculations with the independent triple $w,U,Y$ below are confined to $O$.

\step{The restricted determinant in invariant form.}
Set
\[
h_0=\xi(Y),\qquad z_0=\zeta(Y),\qquad H_0=z_0+2h_0,\qquad
k_0=(\eta-\zeta)/2.
\]
In the frame \eqref{eq:singular-unitary-B},
\[
w=s_0e_1\otimes\nu_0,\quad
U=(A_0e_1+2De_2+2be_3)/s_0,\quad
\zeta=(s_0/b)\varphi_3,\quad\xi=b\varphi_2-D\varphi_3.
\]
Since $W$ has kernel $L$ and $W(U,\cdot)=-2\xi$, independence of $w,U,Y$ gives $h_0\ne0$. The same components give
\begin{equation}\label{eq:W-zeta-xi}
W=-\zeta\wedge\xi.
\end{equation}
The recurrence for $w$ is also the recurrence for $W$, under the parallel alternating-symbol identification. Differentiate $\xi=-\iota_UW/2$, using the full recurrence of $U$, to obtain
\[
\nabla_X\xi=(\eta/2-3\zeta/2)(X)\xi-\tfrac12W(X,\cdot).
\]
Differentiate $h_0=\xi(Y)$ and use $\nabla_XY=a(X)Y$ and \eqref{eq:W-zeta-xi}. The result is
\begin{equation}\label{eq:h-gamma}
\partial h_0=h_0k_0-\tfrac12H_0\xi,\qquad \partial\gamma=\gamma k_0.
\end{equation}

On $E/L$, the vectors $U,Y$ form a basis. Evaluation by $\zeta,\xi$ gives
\[
e_2=\frac b{h_0}Y-\frac{bz_0}{2h_0}U\pmod L.
\]
Substitute this into \eqref{eq:singular-row-factor} with $y_1=0$, using $\ell(y)=s_0U(y)$ and $V=Y-\gamma U$. We obtain
\begin{equation}\label{eq:restricted-invariant}
Q_1|_{y_1=0}=\frac{s_0}{2h_0}
\bigl\{Y(y)^2-(H_0/2+2\gamma)U(y)Y(y)
+\gamma(z_0+h_0)U(y)^2\bigr\}.
\end{equation}
The restricted polynomials $U(y),Y(y)$ are independent. Since the symmetric-matrix determinant of $ax^2+bxy+cy^2$ is $ac-b^2/4$, the zero determinant in Proposition~\ref{prop:singular-conic} gives
\begin{equation}\label{eq:gamma-det}
H_0^2=8\gamma(z_0-2\gamma),\qquad
(H_0-4\gamma)^2=-16\gamma h_0.
\end{equation}
These are complex bilinear squares, without conjugation.

\step{The coefficient $\gamma$ vanishes.}
Work on the open subset of $O$ where $\gamma\ne0$. Put $m_0=H_0/\gamma$ and $r_0=h_0/\gamma$. Equation~\eqref{eq:gamma-det} says $(m_0-4)^2=-16r_0\ne0$. From \eqref{eq:h-gamma},
\[
\partial r_0=-\frac12m_0\xi,\qquad
\partial m_0=\frac{4m_0}{m_0-4}\xi.
\]
Apply $\partial$ again. The derivative of the scalar coefficient in the second equation is proportional to $\partial m_0$, so its wedge with $\xi$ vanishes. Since $\partial\xi=W/2$, we get
\[
0=\frac{2m_0}{m_0-4}W.
\]
Thus $m_0=0$. Equation~\eqref{eq:gamma-det} now gives $z_0=2\gamma$, $h_0=-\gamma$, and \eqref{eq:restricted-invariant} becomes
\begin{equation}\label{eq:restricted-square}
Q_1|_{y_1=0}=-\frac{s_0}{2\gamma}V(y)^2.
\end{equation}
It is a nonzero rank-one quadratic form, and $V\notin L$.

The radical of the conjugate restricted form is the smooth line
\[
L_{\rm rad}=\{X\in H:(X,w)=0,\ (X,V)=0\}.
\]
The first condition simply records $H=L^\perp$, with the canonical factor in $w$ suppressed. We verify that this radical is holomorphic on the present open set. For $X\in L_{\rm rad}$ and $Z\in H$, $Q^c(X,Z)=0$. Differentiate in a barred direction. The output term vanishes by the conjugate of \eqref{eq:Q-transport}; the term containing $\nabla''Z$ vanishes because $H$ is holomorphic and $X$ is in the radical. Hence $Q^c(\nabla''X,Z)=0$ for all $Z\in H$. Since $\nabla''X\in H$, it belongs to the radical. Thus $L_{\rm rad}$ is $\bp$-preserved, and its orthogonal complement $\Span(w,V)$ is preserved by unbarred Chern derivatives.

Take $X=U$ in \eqref{eq:singular-VY}, using $\zeta(U)=2$. It gives
\[
\nabla_UV=a(U)V+\gamma U.
\]
The left side and the first term lie in $\Span(w,V)$, while $U$ does not, since $w,U,Y$ and hence $w,U,V$ are independent. Thus $\gamma=0$, a contradiction. Its nonzero set in $O$ is empty. By continuity and density, $\gamma=0$ on all of $O_*$. Only this smooth scalar is continued; the quotients by $h_0$ or $\gamma$ are not.

\step{The quantitative estimate.}
On $O$, equation \eqref{eq:gamma-det} gives $H_0=0$ and $Y=V$. At a point use \eqref{eq:singular-unitary-B}, and put $z=q/2=s_0/(2b)$ and $k=z-D$. The equation $H_0=0$ is $bV^2+kV^3=0$. Since $b\ne0$, write
\[
V=(x,kc,-bc)
\]
for scalars $x,c$. These are pointwise component names; the frame is not being differentiated here. Equation~\eqref{eq:singular-row-factor} becomes
\begin{equation}\label{eq:final-row-polynomials}
\begin{split}
Q_1(y)&=(-A_0y_1/2+ky_2-by_3)(xy_1+kcy_2-bcy_3),\\
Q_2(y)&=-zy_1(xy_1+kcy_2-bcy_3),\qquad Q_3(y)=0.
\end{split}
\end{equation}
Remembering that the coefficient of a mixed monomial is twice its symmetric-matrix entry, the contraction in \eqref{eq:Q-decomp} gives
\begin{equation}\label{eq:final-G-matrix}
G=\begin{pmatrix}
zbc&0&0\\
bx-A_0bc/2&2kbc&-2b^2c\\
(k+2z)x-A_0kc/2&kc(2k+z)&-bc(2k+z)
\end{pmatrix}.
\end{equation}
Its trace is zero, and direct multiplication gives
\begin{equation}\label{eq:final-G-trace}
\tr G^2=b^2c^2\{z^2+4k^2+(2k+z)^2-4k(2k+z)\}
=2z^2b^2c^2.
\end{equation}
Therefore $r_G\leq2|bc|^2|z|^2$.

Since $S=\B\odot V$, three of its independent entries are
\[
S_{122}=\frac23Dkc,\qquad S_{133}=-\frac23b^2c,\qquad
S_{123}=\frac13bc(k-D).
\]
Their ordered multiplicities are $3,3,6$, respectively. Thus
\begin{equation}\label{eq:final-S-bound}
3|S|^2\geq4|Dkc|^2+4|b^2c|^2+2|bc(k-D)|^2
\geq2|bc|^2|k+D|^2=2|bc|^2|z|^2.
\end{equation}
For the second inequality, subtraction of the right side gives
\[
4|c|^2\{\,|Dk|^2+|b|^4-2|b|^2\re(k\bar D)\,\}\geq0,
\]
by $\re(k\bar D)\leq|Dk|$ and $|Dk|^2+|b|^4\geq2|b|^2|Dk|$. Hence $r_G\leq3|S|^2$ on $O$.

Both sides are global smooth scalar contractions, so this inequality extends to $M$ by density. The tensor $S=\B\odot V$ is nonzero on $O$: $\B\ne0$, $V\ne0$, and symmetric polynomial multiplication has no zero divisors. Proposition~\ref{prop:compact-estimate} therefore gives $\psi\geq0$ and
\[
0\geq6\int_Mt|\Nt|^2+3\int_Mt|S|^2
+\frac83\int_Mt^3+2\int_M\psi t^2\geq0.
\]
It follows that $t=0$ everywhere, contrary to rank two of $\B$ on $O_*$. This proves the proposition. The final integrals involve only smooth tensors on the whole compact manifold, including every rank drop and determinant zero.
\end{proof}

\vspace{0.3cm}

\section{Proof of Theorem \ref{thm1}}\label{sec:conclusion}
\begin{proof}[{\bf Proof of Theorem~\ref{thm1}}]

Let $(M^3,g)$ be a compact Hermitian manifold of complex dimension three with vanishing Chern holomorphic sectional curvature. Assume that the Chern curvature tensor is not identically zero, namely, $R\not\equiv0$. Proposition~\ref{prop:Qzero} excludes the case $Q=0$. Proposition~\ref{prop:rank} implies that the rank of $Q$ is two, and Lemma~\ref{lem:continuation} makes its regular locus open and dense. The pure-trace case is excluded by Lemma~\ref{lem:pure-trace}. Lemma~\ref{lem:cubic-dense} and Propositions~\ref{prop:zero-curl}--\ref{prop:flag} then supply the nonempty dense open set $O_*$.

The global section $\delta=\det\B$ is either identically zero or not. If $\delta\not\equiv0$, Proposition~\ref{prop:normalization} constructs the nonempty regular set $U$ and its dense nonzero-acceleration subset. Then Proposition~\ref{prop:full-flat} gives us a contradiction there. If $\delta\equiv0$, Propositions~\ref{prop:singular-conic} and \ref{prop:singular-recurrence} give the holomorphic quotient and the full recurrence on $O_*$. Proposition~\ref{prop:singular-flat} then gives a global scalar bound contradicting the integral inequality. Both alternatives are impossible. Hence we must have $R=0$ on $M$, and the proof of Theorem \ref{thm1} is completed.
\end{proof}

Let us finish the discussion by clarifying the role of the exceptional sets and of the quadratic-form ranks in the proof.

First of all, the determinant alternatives above are global alternatives for the section $\delta$. They are not the pointwise partition $\{\delta=0\}\cup\{\delta\ne0\}$. In the case $\delta\not\equiv0$, the section $\delta$ may still have zeros. The proof uses only a nonempty regular open set on which a contradiction is obtained; it does not extend $F_i$, $N_0$, $\mu$, or $\B^{-1}$ across those zeros.

Second, neither the source-zero case nor the lower-rank loci are omitted in the integral argument. The tensors in \eqref{eq:scalar} and \eqref{eq:weighted-energy} are smooth on $M$. If $3|S|^2+2|\Nt|^2$ were identically zero, then $S=0$, and Lemma~\ref{lem:cubic-dense} would give $R=0$. Individual zeros of the source remain in the negative-part test and in every integral. In the singular case, $\B$ has rank two only on $O_*$; lower ranks outside that set are allowed and are retained in the final global inequality.

Third, the extension arguments have distinct domains. The exterior-power uniqueness in Lemma~\ref{lem:continuation} uses a smooth barred system. After Lemma~\ref{lem:analytic}, analytic continuation is applied only to global analytic tensors such as $\B$, $W$, $G^2$, and the polynomial detectors. The form $\beta$ is extended by its analytic square relation. The quotient map $\alpha_0$ may vanish on $\Sigma$, and its kernel is not asserted to extend there as a plane bundle. Only the determinant of the restricted quadratic form is extended across $\Sigma$, by meromorphic descent and removal of divisorial poles. Finally, the scalar estimate $r_G\leq3|S|^2$ extends by continuity of its two globally smooth sides. The normalized fields used to establish it need not extend.

\begin{remark}
The output rank of $Q$ is different from the matrix rank of any individual quadratic row. On $O_2$, choose a basis of the two-dimensional output and write $Q^c=Aq_1+Bq_2$, where $q_1,q_2$ are independent scalar quadratic forms. Their common projective roots, common radical, and the ranks of the pencil $sq_1+tq_2$ are different notions. Neither a simple common root nor a nonsingular member of this pencil is assumed in the argument. In Section~\ref{sec:full}, the particular root $F_1$ is derived before it is differentiated. In Section~\ref{sec:singular}, the radical is used only after the restricted quadratic form is proved to be a nonzero square on the open set $\{\gamma\ne0\}$.

For comparison, a linear space of symmetric matrices in which every member has rank at most one has dimension at most one. A nonzero rank-one symmetric form is a scalar square; two independent such forms are $cl^2,dm^2$ with independent linear forms $l,m$, and their sum has rank two. This observation concerns the rank of every member of a pencil, not the output rank of $Q$ or the rank of a selected restriction. These descriptions therefore introduce no additional global case beyond the determinant alternatives in the proof.
\end{remark}

\vspace{0.3cm}

\noindent\textbf{Generative AI disclosure.}
The authors used OpenAI's ChatGPT as an assisting tool in the development of the paper. The authors take full responsibility for the contents of the paper.

\vspace{0.3cm}

\noindent\textbf{Declaration of competing interests.}
The authors declare that they have no competing interests.

\vspace{0.3cm}

\noindent\textbf{Statement on data availability.}
Data sharing is not applicable to this article as no datasets were generated during the current study.

\end{document}